\documentclass[reqno]{amsart}

\usepackage[T1]{fontenc}
\usepackage{amsmath}
\usepackage{amssymb}
\usepackage{amsthm}
\usepackage{hyperref}
\usepackage{fancyhdr}
\usepackage[textsize=small]{todonotes}
\usepackage{enumitem}
\usepackage{comment}
\usepackage{nicefrac}
\usepackage{bm}
\usepackage{mathrsfs}
\usepackage{graphicx}
\usepackage[utf8]{inputenc}
\usepackage{cancel}
\usepackage{mathtools}
\usepackage{tikz}
\usetikzlibrary{cd,decorations.pathreplacing,matrix,arrows,positioning,automata,shapes,shadows,calc,fadings,decorations,snakes,through,intersections}
\usepackage{float}

\theoremstyle{plain}
\newtheorem{theorem}{Theorem}[section]
\newtheorem{lemma}[theorem]{Lemma}
\newtheorem{corollary}[theorem]{Corollary}
\newtheorem{proposition}[theorem]{Proposition}

\theoremstyle{definition}
\newtheorem{definition}[theorem]{Definition}

\newtheorem{example}[theorem]{Example}
\newtheorem{remark}[theorem]{Remark}

\newtheorem{question}[theorem]{Question}

\newtheorem{claim}{Claim}

\newcommand{\bN}{\mathbb{N}}
\newcommand{\N}{\bN}

\newcommand{\cA}{\mathcal{A}}
\newcommand{\cB}{\mathcal{B}}

\newcommand{\cI}{\mathcal{I}}
\newcommand{\I}{\cI}
\newcommand{\cJ}{\mathcal{J}}
\newcommand{\J}{\cJ}

\newcommand{\fin}{\mathrm{Fin}}
\newcommand{\Fin}{\fin}

\begin{document}


\title{Characterized subgroups on the unit circle}


\author[R.~Filip\'{o}w]{Rafa\l{} Filip\'{o}w}
\address[Rafa\l{}~Filip\'{o}w]{Institute of Mathematics\\ Faculty of Mathematics, Physics and Informatics\\ University of Gda\'{n}sk\\ ul.~Wita Stwosza 57\\ 80-308 Gda\'{n}sk\\ Poland}
\email{Rafal.Filipow@ug.edu.pl}
\urladdr{\url{http://mat.ug.edu.pl/~rfilipow}}

\author[A.~Kwela]{Adam Kwela}
\address[Adam Kwela]{Institute of Mathematics\\ Faculty of Mathematics\\ Physics and Informatics\\ University of Gda\'{n}sk\\ ul.~Wita  Stwosza 57\\ 80-308 Gda\'{n}sk\\ Poland}
\email{Adam.Kwela@ug.edu.pl}
\urladdr{\url{https://mat.ug.edu.pl/~akwela}}

\author[P.~Leonetti]{Paolo Leonetti}
\address[Paolo Leonetti]{Department of Economics\\ Universit\'{a} degli Studi dell’Insubria\\ via Monte Generoso 71 \\ Varese 21100\\ Italy}
\email{leonetti.paolo@gmail.com}
\urladdr{\url{https://sites.google.com/site/leonettipaolo}}

\author[J.~Tryba]{Jacek Tryba}
\address[Jacek Tryba]{Institute of Mathematics\\ Faculty of Mathematics\\ Physics and Informatics\\ University of Gda\'{n}sk\\ ul.~Wita  Stwosza 57\\ 80-308 Gda\'{n}sk\\ Poland}
\email{Jacek.Tryba@ug.edu.pl}










\begin{abstract}
\noindent Given an ideal $\mathcal I$ on $\omega$, a subgroup $H$ of the unit circle $\mathbb{T}$ is said to be $\mathcal{I}$-characterized if there exists an integer sequence $\bm{a}=(a_n: n \in \omega)$ such that
$$
H=\mathsf H_{\bm a}(\mathcal I):=
\left\{x\in\mathbb T:\mathcal I\text{-}\lim_{n\to \infty} a_nx=0\right\}.
$$
We also consider the corresponding \(\mathcal I^\star\)-version. We provide upper bounds for the
topological complexities of those subgroups in terms of the complexity of $\mathcal I$. 

Moreover, we prove that Rudin--Keisler and Rudin--Blass reductions between ideals induce inclusions between the corresponding families of characterized subgroups. As a
consequence, every characterized subgroup, and in particular every countable
subgroup of $\mathbb T$, is $\mathcal I$-characterized for every meager
ideal $\mathcal I$. 
We also show that if the image of $(a_n: n \in \omega)$ contains arbitrarily large intervals, then every subgroup of $\mathbb{T}$ can be written as $\mathsf H_{\bm a}(\mathcal J)$ for some ideal $\mathcal{J}=\mathcal{J}_{H,\bm{a}}$. 
We analyze the descriptive complexity and $P$-properties of these ideals. 

Finally, we study when the equality $\mathsf H_{\bm a}(\mathcal I)=\mathbb T$ forces $\mathrm{supp}(\bm a)\in\mathcal I$. 
We prove this for a class of ideals satisfying a Katětov-type condition involving $\mathcal{ED}$, including nowhere tall ideals as well as the ideals $\mathsf{nwd}$ and $\mathsf{null}$. 
We also obtain non-inclusion results between families of $\mathcal{I}$-characterized subgroups: for instance, we show that if the ideal $\mathcal{I}$ is tall and translation invariant then the subgroup $\mathsf H_{(2^n)}(\mathcal I)$ cannot be characterized.

We use our results to answer several open problems posed in the literature. 
\end{abstract}


\maketitle



\thispagestyle{empty}

\section{Introduction}\label{sec:intro}

Let $\mathbb{T}:=\mathbb{R}/\mathbb{Z}$ be the unit circle. By a usual abuse of notation, we identify each real $x \in \mathbb{R}$ with its equivalence class $x+\mathbb{Z}$, and we write $\|x\|:=\min\{|x-z|:z \in \mathbb{Z}\}$. In particular, $\mathbb{T}$ is a compact Polish group. 
Let us recall that a map $\chi: \mathbb{T}\to \mathbb{T}$ is a character (that is, a continuous homomorphism) if and only if there exists $a \in \mathbb{Z}$ such that $\chi(x):=ax$ for all $x \in \mathbb{T}$. In particular, a sequence of characters $(\chi_n: n\in \omega)$ can be uniquely represented by a sequence of integers $\bm{a}=(a_n: n \in \omega) \in \mathbb{Z}^\omega$. 

Let $\mathcal{I}$ be an ideal on the nonnegative integers, that is, a family of subsets of $\omega$ stable under finite unions and subsets. Unless otherwise stated, it is always assumed that the family of finite sets $\mathrm{Fin}:=[\omega]^{<\omega}$ is contained in $\mathcal{I}$, and that $\mathcal{I}$ is proper, that is, $\omega\notin \mathcal{I}$. Set  $\mathcal{I}^+:=\mathcal{P}(\omega)\setminus \mathcal{I}$. In place of $\omega$, we will consider also ideals on countably infinite sets. 
Among important examples of ideals are the family $\mathcal{Z}$ of asymptotic density zero sets and maximal ideals (that is, the complements of free ultrafilters on $\omega$). 
An ideal $\mathcal{I}$ is said to be a $P$-ideal if for every family $\{S_n: n \in \omega\}$ of sets in $\mathcal{I}$ there exists $S \in \mathcal{I}$ such that $S_n\setminus S$ is finite for all $n \in \omega$. 
Identifying $\mathcal{P}(\omega)$ with the Cantor space, we can speak about the topological complexity of ideals. 
For instance, $\mathcal{Z}$ is an $F_{\sigma\delta}$ $P$-ideal which is not $F_\sigma$. 
We refer the reader to \cite{MR1711328} for an excellent textbook on the theory of ideals on $\omega$.

Let us recall that a sequence $\bm{x}=(x_n: n \in \omega) \in \mathbb{T}^\omega$ is $\mathcal{I}$-convergent to $\eta \in \mathbb{T}$, shortened as $\mathcal{I}\text{-}\lim \bm{x}=\eta$ or $\mathcal{I}\text{-}\lim_n x_n=\eta$, if $\{n \in \omega: \|x_n-\eta\|\ge \varepsilon\} \in \mathcal{I}$ for all $\varepsilon>0$. Of course, $\mathrm{Fin}$-convergence coincides with ordinary convergence. Given a sequence of integers $\bm{a}=(a_n: n \in \omega)$, define the subgroup 
\begin{equation}\label{eq:defHaI}
\mathsf{H}_{\bm{a}}(\mathcal{I}):=\left\{x \in \mathbb{T}: \mathcal{I}\text{-}\lim_{n\to \infty} a_nx=0\right\}.
\end{equation}
Note that the above definition can be extended along the same lines by replacing the circle group $\mathbb{T}$ with an arbitrary topological group $X$, and the sequence of integers $(a_n:n\in\omega)$ with a sequence of characters $(\chi_n:n\in\omega)$ from $X$ to $\mathbb{T}$. In this work, we restrict ourselves to the circle group mainly for the sake of simplicity of exposition. 

It is readily seen that, if $\bm{q}=(q_n: n \in \omega)$ is an increasing sequence of positive integers, then $\mathsf{H}_{\bm{q}}(\mathrm{Fin})$ might be interpreted as the family of reals with a sufficiently good rational approximation along the prescribed sequence of denominators $(q_n: n \in \omega)$. In fact, pick $x \in \mathbb{T}$ and, for each $n \in \omega$, let $p_{n}$ be an integer which minimizes $\|x-\frac{p_n}{q_n}\|$. Then 
$$
x \in \mathsf{H}_{\bm{q}}(\mathrm{Fin})
\quad\text{ if and only if }\quad 
\left\|x-\frac{p_n}{q_n}\right\|=o\left(\frac{1}{q_n}\right) \text{ as }n\to \infty.
$$
\begin{definition}\label{def:Ichar}
    Let $\mathcal{I}$ be an ideal on $\omega$. A subgroup $H$ of $\mathbb{T}$ is said to be $\mathcal{I}$\emph{-characterized} if there exists a sequence of integers $\bm{a}=(a_n: n \in \omega)$ such that 
    $$
    H=\mathsf{H}_{\bm{a}}(\mathcal{I}).
    $$ 
    In the case $\mathcal{I}=\mathrm{Fin}$, the subgroup $H$ is said to be \emph{characterized}. In addition, we denote the family of $\mathcal{I}$-characterized subgroups by 
    $$
    \mathscr{H}(\mathcal{I}):=\{\mathsf{H}_{\bm{a}}(\mathcal{I}): \bm{a} \in \mathbb{Z}^\omega\}.
    $$
\end{definition}

The study of characterized subgroups is rooted in the broader problem of understanding how the algebraic structure of a topological group is reflected by convergence phenomena. 
Already in the classical theory of topological groups, topologically torsion elements were introduced independently by Braconnier \cite{MR13158} and Vilenkin \cite{MR14104} and played an important role in the structure theory of locally compact abelian groups and profinite groups. 
From this point of view, a characterized subgroup of $\mathbb T$ is not an ad hoc object: it is the subgroup of elements which are ``topologically torsion'' with respect to a prescribed sequence of characters, cf. \cite{MR346087}. 
Thus the condition $\lim_n a_nx=0$, which goes in the opposite direction of equidistribution-type results, gives a concrete way to test algebraic information by asymptotic behaviour in the dual group. 
This makes the class of characterized subgroups rigid enough to have strong structural and regularity properties, but flexible enough to contain many nontrivial examples: 
for instance, it was proved by Erd\H{o}s and Taylor \cite{MR92032} that if a positive sequence $\bm{a}$ satisfies $\sup_n a_{n+1}/a_n<\infty$ then $\mathsf{H}_{\bm{a}}(\mathrm{Fin})$ is countable; vice versa, 
B\'ir\'o, Deshouillers and S\'os \cite{MR1877772} 
showed that every countable subgroup of $\mathbb{T}$ is characterized. 
Hence they provide a natural meeting point between harmonic analysis, Diophantine approximation, and descriptive set theory. 
Lastly, it is worth remarking that alternative constructions of subgroups which involve ideals on $\omega$ in the same spirit of \eqref{eq:defHaI} are known to be fruitful, see e.g. Farah and Solecki \cite[p. 516]{MR2189217}, Solecki \cite[Section 3]{MR2197115}, and Hu and Solecki \cite[Section 2]{SoleckiHu}. 
We refer the reader to \cite{MR3864800, MR2032835, MR4932230, MR3145818, MR3288121, MR2921827} for detailed surveys on the history and motivations on characterized subgroups and their ideal version. 

In a similar direction, recall that a sequence $\bm{x}=(x_n: n \in \omega) \in \mathbb{T}^\omega$ is $\mathcal{I}^\star$-convergent to $\eta\in \mathbb{T}$, shortened as $\mathcal{I}^\star\text{-}\lim \bm{x}=\eta$ or $\mathcal{I}^\star\text{-}\lim_n x_n=\eta$, if there exists $S \in\mathcal{I}$ such that $(x_n: n \in \omega\setminus S)$ is convergent (in the ordinary sense) to $\eta$. Mimicking the definition in \eqref{eq:defHaI}, define also the subgroup
$$
\mathsf{H}^\star_{\bm{a}}(\mathcal{I}):=
\left\{x \in \mathbb{T}: \mathcal{I}^\star\text{-}\lim_{n\to \infty} a_nx=0\right\}.
$$

\begin{definition}\label{def:Istarchar}
    Let $\mathcal{I}$ be an ideal on $\omega$. A subgroup $H$ of $\mathbb{T}$ is said to be $\mathcal{I}^\star$\emph{-characterized} if there exists a sequence of integers $\bm{a}=(a_n: n \in \omega)$ such that  
    $$
    H=\mathsf{H}^\star_{\bm{a}}(\mathcal{I}).
    $$ 
    We denote the family of $\mathcal{I}^\star$-characterized subgroups by 
    $$
    \mathscr{H}^\star(\mathcal{I}):=\{\mathsf{H}^\star_{\bm{a}}(\mathcal{I}): \bm{a} \in \mathbb{Z}^\omega\}.
    $$
\end{definition}

Since $\mathcal{I}^\star$-convergence implies $\mathcal{I}$-convergence, we obtain that $\mathsf{H}^\star_{\bm{a}}(\mathcal{I})$ is a subgroup of $\mathsf{H}_{\bm{a}}(\mathcal{I})$. 
On the other hand, if $\mathcal{I}$ is a $P$-ideal, then $\mathcal{I}^\star$-convergence coincides with $\mathcal{I}$-convergence, so  $\mathsf{H}^\star_{\bm{a}}(\mathcal{I})= \mathsf{H}_{\bm{a}}(\mathcal{I})$, see e.g. \cite[Theorem 2.4(iii)]{MR3920799}. 

Some instances of our main results follow below: 
\begin{enumerate}[label={\rm (\arabic*)}]
\item [(a)] $\mathcal{I}$-characterized subgroups coincide with the corresponding  $\mathcal{I}^\star$-characterized subgroups if and only if $\mathcal{I}$ is a $P$-ideal (see Theorem \ref{thm:Pproperty});
\item [(b)] $\mathcal{I}$-characterized subgroups are Borel whenever $\mathcal{I}$ is Borel (see Theorem \ref{thm:upperboundscomplexity}); 
\item [(c)] if $\mathcal{I}$ is meager, then every characterized subgroup is also $\mathcal{I}$-characterized (see Corollary \ref{RB:corollary}); 
\item [(d)] If $\bm{a}$ is a sequence of naturals whose image contains arbitrarily large intervals, then every subgroup is $\mathcal{I}$-characterized through the same sequence $\bm{a}$ for some ideal $\mathcal{I}$ (see Theorem \ref{thm:beigblock_representation}); 
\item [(e)] If $\mathcal{I}$ is a nowhere tall ideal, then $\mathsf H_{\bm a}(\mathcal I)=\mathbb T$ if and only if
$\operatorname{supp}(\bm a)\in\mathcal I$ (see Corollary \ref{cor:1P^-A}); 
\item [(f)] For every tall translation invariant ideal $\mathcal{I}$ there exists an $\mathcal{I}$-characterized subgroup which cannot be characterized (see Corollary \ref{cor:Hpowernotfin}).
\end{enumerate}

The paper is organized as follows. In Section \ref{sec:main} we present a summary of our main results (without proofs) and use them to answer several open problems posed in the literature. Section \ref{sec:preliminaries} is devoted to some preliminary results. The proofs of our main results can be found in Section \ref{sec:proofs}.

\section{Main results}\label{sec:main}

\subsection{Preliminary results}

Hereafter, we will consider the \textquotedblleft universal\textquotedblright\, sequence 
$$
\bm{u}=(u_n: n\in \omega)
\quad \text{ defined by }\quad 
u_n:=n \text{ for all }n \in \omega.
$$ 
As mentioned in Section \ref{sec:intro}, we will show in Theorem \ref{thm:beigblock_representation} below that \emph{every} subgroup of $\mathbb{T}$ can be represented as $\mathsf{H}_{\bm{u}}(\mathcal{I})$ for some ideal $\mathcal{I}$ on $\omega$, cf. also \cite[Theorem 2.1]{MR2227021}. On the other hand, with the exception of the trivial subgroup $\{0\}$, Example \ref{example:singletonzero} shows that such ideals cannot be contained in $\mathcal{Z}$, cf. Lemma \ref{lem:basicIH}\ref{item:4basiclemma}.
\begin{example}\label{example:singletonzero}
    Pick an ideal $\mathcal{I}$ on $\omega$ such that $\mathcal{I}\subseteq \mathcal{Z}$. Then 
    $
    \mathsf{H}^\star_{\bm{u}}(\mathcal{I})=\mathsf{H}_{\bm{u}}(\mathcal{I})=\{0\}.
    $ 
    
    Indeed, taking into account that $\mathsf{H}^\star_{\bm{u}}(\mathcal{I})\le \mathsf{H}_{\bm{u}}(\mathcal{I})\le \mathsf{H}_{\bm{u}}(\mathcal{Z})$, it is enough to show that $\mathsf{H}_{\bm{u}}(\mathcal{Z})=\{0\}$. 
    To this aim, observe that if $x=\nicefrac{p}{q} \in \mathbb{T}$ is rational with $\mathrm{gcd}(p,q)=1$ and $q\ge 2$, then $\{n \in \omega: \|nx\|\ge \nicefrac{1}{q}\}$ admits positive asymptotic density $1-\nicefrac{1}{q}$. At the same time, if $x\in \mathbb{T}$ is irrational then $(\|nx\|: n \in \omega)$ is equidistributed in $\mathbb{T}$ by the Kronecker--Weyl theorem, hence not $\mathcal{Z}$-convergent to $0$. Therefore $\mathsf{H}_{\bm{u}}(\mathcal{Z})=\{0\}$. 
\end{example}

Before we proceed to our main results, some remarks are in order. Let us recall that $\eta \in \mathbb{T}$ is an $\mathcal{I}$\emph{-cluster point} of a sequence $\bm{x}=(x_n: n \in \omega) \in \mathbb{T}^\omega$ if $\{n \in \omega: \|x_n-\eta\|\le  \varepsilon\} \notin \mathcal{I}$ for all $\varepsilon>0$. The set of $\mathcal{I}$-cluster points of $\bm{x}$ is denoted by $\Gamma_{\bm{x}}(\mathcal{I})$. Accordingly, it is easy to see that for each sequence of integers $\bm{a}=(a_n: n \in \omega) \in \mathbb{Z}^\omega$ we have
    $$
    \mathsf{H}_{\bm{a}}(\mathcal{I})
    =\{x \in \mathbb{T}: \Gamma_{\bm{a}x}(\mathcal{I})\subseteq \{0\}\}
    =\{x \in \mathbb{T}: \Gamma_{\bm{a}x}(\mathcal{I})= \{0\}\},
    $$
    where $\bm{a}x:=(a_nx: n \in \omega)$; cf. \cite[Lemma 3.1(vi) and Lemma 3.3]{MR3920799}.

However, an analogous representation does not hold for $\mathcal{I}$-limit points; here we recall that $\eta \in \mathbb{T}$ is an $\mathcal{I}$\emph{-limit point} of a sequence $\bm{x} \in \mathbb{T}^\omega$ if there exists $S \in \mathcal{I}^+$ such that $\lim_{n \in  S}x_n=\eta$. The set of $\mathcal{I}$-limit points of $\bm{x}$ is denoted by $\Lambda_{\bm{x}}(\mathcal{I})$. In the example below, we show that the analogue subsets of $\mathcal{I}$-limit points might not be subgroups. 

\begin{example}\label{example:Ilimitpoints} 
    Given a sequence $\bm{a}$ that we will define below, set 
    $$
    H_1:=\{x \in \mathbb{T}: \Lambda_{\bm{a}x}(\mathcal{Z})\subseteq \{0\}\}
    \quad \text{ and }\quad 
    H_2:=\{x \in \mathbb{T}: \Lambda_{\bm{a}x}(\mathcal{Z})= \{0\}\}.
    $$
    Let $(d_k: k \in \omega)$ be a uniformly distributed sequence in $\mathbb{T}$, pick a nonzero $c \in \mathbb{T}$ and define $v_{2k}:=d_k$, $w_{2k}:=d_k-c$, and $v_{2k+1}:=w_{2k+1}:=0$ for all $k \in \omega$. Then it is easy to check that $\Lambda_{\bm{v}}(\mathcal{Z})=\Lambda_{\bm{w}}(\mathcal{Z})=\{0\}$, cf. \cite[Example 4]{MR1181163}. Now, pick $x,y \in \mathbb{T}$ such that $\{1,x,y\}$ is $\mathbb{Q}$-independent. By the Kronecker--Weyl theorem, the sequence $((nx,ny): n \in \omega)$ is dense in $\mathbb{T}^2$. Hence there is a strictly increasing sequence $\bm{a} \in \omega^\omega$ such that 
    $$
    \|a_nx-v_n\|<2^{-n} 
    \quad \text{ and }\quad 
    \|a_ny-w_n\|<2^{-n}
    $$
    for all $n \in \omega$. 
    It follows from \cite[Lemma 3.5(ii)]{MR3920799} that $\Lambda_{\bm{a}x}(\mathcal{Z})=\Lambda_{\bm{a}y}(\mathcal{Z})=\{0\}$. Taking into account that $\bm{v}-\bm{w}=(c,0,c,0,\ldots)$, we obtain by the same reasoning that $\Lambda_{\bm{a}(x-y)}(\mathcal{Z})=\Lambda_{\bm{v}-\bm{w}}(\mathcal{Z})=\{0,c\}$. Therefore both $H_1$ and $H_2$ are not subgroups of $\mathbb{T}$.
\end{example}


\subsection{A characterization of the $P$-property} 
In our first result, we prove that the equality between the subgroups $\mathsf{H}^\star_{\bm{a}}(\mathcal{I})$ and $\mathsf{H}_{\bm{a}}(\mathcal{I})$ for every integer sequence $\bm{a}$ characterizes the $P$-property of the ideal $\mathcal{I}$.
\begin{theorem}\label{thm:Pproperty}
    Let $\mathcal{I}$ be an ideal on $\omega$. 
    Then the following are equivalent\textup{:}
    \begin{enumerate}[label={\rm (\roman*)}]
    \item\label{item:01Pidealchar} $\mathcal{I}$ is a $P$-ideal\textup{;}
    \item\label{item:02Pidealchar} $\mathsf{H}_{\bm{a}}(\mathcal{I})=\mathsf{H}^\star_{\bm{a}}(\mathcal{I})$ for all sequences $\bm{a} \in \mathbb{Z}^\omega$\textup{;}
    \item\label{item:03Pidealchar} $\mathsf{H}_{\bm{a}}(\mathcal{I})=\mathsf{H}^\star_{\bm{a}}(\mathcal{I})$ for some sequence $\bm{a} \in (\omega\setminus \{0\})^\omega$ such that $\lim_n a_{n}/a_{n+1}=0$\textup{.}
    \end{enumerate}
\end{theorem}

The proof of Theorem \ref{thm:Pproperty} relies on the following property (which will be proved in Proposition \ref{prop:interpolation} below): for each sequence $\bm{a}$ as in item \ref{item:03Pidealchar} and for each sequence $(t_n:n\in\omega)$ in $\mathbb T$, there exists $x\in\mathbb T$ such that $\lim_{n} (a_nx-t_n)=0$. 

We do not know whether the weaker equality $\mathscr{H}(\mathcal{I})=\mathscr{H}^\star(\mathcal{I})$ characterizes the $P$-property as well, hence we leave it as an open question.
\begin{question}
    Let $\mathcal{I}$ be an ideal on $\omega$. Is it true that $\mathcal{I}$ is a $P$-ideal if and only if $\mathscr{H}(\mathcal{I})=\mathscr{H}^\star(\mathcal{I})$?
\end{question}





\subsection{Upper bounds on topological complexities}

In the next result, following \cite[p. 199]{MR2849045}, an ideal $\I$ on $\omega$ is said to be a \emph{Farah ideal} if there exists a sequence $(C_i: i \in \omega)$ of hereditary\footnote{A family $\mathcal{F}\subseteq \mathcal{P}(\omega)$ is said to be hereditary if $A \in \mathcal{F}$ for all $A,B\subseteq \omega$ with $A\subseteq B\in \mathcal{F}$.} compact subsets of $\mathcal{P}(\omega)$ such that 
\begin{equation}\label{eq:Farahrepresentation}
\I=\bigcap_{i\in\omega}\bigcup_{n\in\omega}\{A\subseteq\omega:A\setminus n\in C_i\}.
\end{equation}
It is well known that if $\I$ is $F_\sigma$ or an analytic $P$-ideal then $\I$ is a Farah ideal, and that every Farah ideal is $F_{\sigma\delta}$ (We recall that it is an open question whether every $F_{\sigma\delta}$ ideal on $\omega$ is a Farah ideal.)
\begin{theorem}\label{thm:upperboundscomplexity}
    Let $\mathcal{I}$ be an ideal on $\omega$ and pick  an integer sequence $\bm{a} \in \mathbb{Z}^\omega$ and a countable ordinal $1\le \alpha<\omega_1$. Then the following hold\textup{:}
        \begin{enumerate}[label={\rm (\roman*)}]
        \item \label{item:1upperbound} If $\mathcal{I}$ is a Farah ideal \textup{(}in particular, if $\I$ is $F_\sigma$ or an analytic $P$-ideal\textup{)} then $\mathsf{H}_{\bm{a}}(\mathcal{I})$ is an $F_{\sigma\delta}$ subgroup\textup{;}
        \item \label{item:2upperbound} If $\mathcal{I}$ is a $\Pi^0_\alpha$ ideal then $\mathsf{H}_{\bm{a}}(\mathcal{I})$ is a $\Pi^0_{\alpha+1}$ subgroup\textup{;}
        \item \label{item:3upperbound} If $\mathcal{I}$ is a $\Sigma^0_\alpha$ ideal then $\mathsf{H}_{\bm{a}}(\mathcal{I})$ is a $\Pi^0_{\alpha+2}$ subgroup\textup{;}
        \item \label{item:4upperbound} If $\mathcal{I}$ is Borel \textup{[}analytic, co-analytic, respectively\textup{]} then $\mathsf{H}_{\bm{a}}(\mathcal{I})$ is a Borel \textup{[}analytic, co-analytic, respectively\textup{]} subgroup\textup{;} 
        \item \label{item:5upperbound} If $\mathcal{I}$ is an analytic ideal then $\mathsf{H}^\star_{\bm{a}}(\mathcal{I})$ is an analytic subgroup\textup{.}
        \end{enumerate}
\end{theorem}

In particular, it follows from item \ref{item:4upperbound} and \cite[Theorem 13.6]{MR1321597} that, if $\mathcal{I}$ is Borel then each subgroup $\mathsf{H}_{\bm{a}}(\mathcal{I})$ is either countable or has cardinality $\mathfrak{c}$.

We remark that special instances of Theorem \ref{thm:upperboundscomplexity}\ref{item:1upperbound} (for certain analytic $P$-ideals) appeared in \cite{protasov, MR1240629, MR4078214}. 
In addition, the analogue of Theorem \ref{thm:upperboundscomplexity} which replaces the unit circle $\mathbb{T}$ with a Polish topological group $X$ and the sequence $\bm{a} \in \mathbb{Z}^\omega$ with a sequence of characters $v_n: X\to \mathbb{T}$ holds verbatim; more generally, item \ref{item:1upperbound} holds for arbitrary topological groups $X$, extending \cite[Theorem 6.5]{protasov} and answering \cite[Problem 6.11]{MR4078214} and \cite[Problem 7.27]{MR4932230} (we omit further details). 

\subsection{When every $\mathcal{I}$-characterized subgroup is $\mathcal{J}$-characterized?}

Properties of ideals can be often expressed by finding critical ideals with respect to some (pre-)order, cf. e.g. the survey \cite{MR2777744}.  
To this aim, given two ideals $\mathcal{I}$ and $\mathcal{J}$ on countably infinite sets $X$ and $Y$, respectively, we say that:
\begin{itemize}
\item (Rudin--Blass order) $\mathcal{I} \le_{\mathrm{RB}} \mathcal{J}$ if there is a finite-to-one function $f: Y\to X$ such that $f^{-1}[I] \in \mathcal{J}$ if and only if $I \in \mathcal{I}$;
\item (Rudin--Keisler order) $\mathcal{I} \le_{\mathrm{RK}} \mathcal{J}$ if there is a function $f: Y\to X$ such that $f^{-1}[I] \in \mathcal{J}$ if and only if $I \in \mathcal{I}$.
\end{itemize}
We write $\mathcal{I} \cong_{\mathrm{RB}} \mathcal{J}$ for $\mathcal{I} \le_{\mathrm{RB}} \mathcal{J} \le_{\mathrm{RB}} \mathcal{I}$, similarly for $\cong_{\mathrm{RK}}$. 
It is clear that $\mathcal{I} \le_{\mathrm{RB}} \mathcal{J}$ implies $\mathcal{I} \le_{\mathrm{RK}} \mathcal{J}$. 
Rudin--Blass and Rudin--Keisler orderings on the maximal ideals are extensively studied in the literature, cf. \cite[Section 1.3]{MR1711328} and references therein. In addition, if $\mathcal{J}$ is a $P$-ideal then $\le_{\mathrm{RB}}$ coincides with $\le_{\mathrm{RK}}$, see \cite[Proposition 1.3.1(b)]{MR1711328}.

In the following, we provide a sufficient condition for which every $\mathcal{I}$-characterized subgroup is also $\mathcal{J}$-characterized. 
\begin{theorem}\label{thm:RKinclusion}
Let $\mathcal{I}, \mathcal{J}$ be ideals on $\omega$. Then the following hold\textup{:}
\begin{enumerate}[label={\rm (\roman*)}]
        \item \label{item:1RK} If $\mathcal{I} \le_{\mathrm{RK}} \mathcal{J}$ then $\mathscr{H}(\mathcal{I})\subseteq \mathscr{H}(\mathcal{J})$\textup{;}
        \item \label{item:2RK} If $\mathcal{I} \le_{\mathrm{RB}} \mathcal{J}$ then $\mathscr{H}^\star(\mathcal{I})\subseteq \mathscr{H}^\star(\mathcal{J})$\textup{.}
\end{enumerate}
\end{theorem}
In particular, it is immediate that $\mathscr{H}(\I)=\mathscr{H}(\J)$ whenever $\mathcal{I} \cong_{\mathrm{RK}} \mathcal{J}$. 
In the following, we denote by $\mathsf{Sub}_{\le \omega}$ the family of countable subgroups of $\mathbb{T}$. 
\begin{corollary}\label{RB:corollary}
    Let $\mathcal{I}$ be a meager ideal on $\omega$. Then 
    $$
    \mathsf{Sub}_{\le \omega}\subseteq 
    \mathscr{H}(\mathrm{Fin})\subseteq 
    \mathscr{H}(\mathcal{I}) \cap \mathscr{H}^\star(\mathcal{I}).
    $$
    In particular, every characterized subgroup of $\mathbb{T}$ is also $\mathcal{I}$-characterized. 
\end{corollary}

As a special case, since $\mathcal{Z}$ is meager, it follows that $\mathsf{Sub}_{\le \omega}\subseteq \mathscr{H}(\mathcal{Z})$, which provides an answer to \cite[Question 7.12 and Question 7.13]{MR4932230},  \cite[Question 6.7]{MR4078214}, and the first part of \cite[Question 6.3]{MR4173319}. In turn, this extends the recent result \cite[Section 3.2]{DasGhoshAziz2025JLMS} that $\mathscr{H}(\mathcal{Z})$ contains some countably infinite subgroup of $\mathbb{T}$, cf. also \cite{MR5000582}. 


\subsection{Every subgroup is $\I$-characterized for some ideal $\I$}

Hereafter, for each sequence $\bm{a} \in \mathbb{Z}^\omega$, we denote its support and image by 
$$
\mathrm{supp}(\bm{a}):=\{n \in \omega: a_n\neq 0\} 
\,\, \text{ and }\,\,
\mathrm{ran}(\bm{a}):=
\{a_n:n\in \omega\}.
$$

Also, for each sequence $\bm{a}=(a_n: n \in \omega) \in \mathbb{Z}^\omega$, and for each $\alpha \in \mathbb{T}$ and $\varepsilon>0$, define 
$$
N_{\bm{a}}(\alpha,\varepsilon):=\{a_n 
: \|a_n\alpha\|\ge \varepsilon\}.
$$
The definition of $N_{a}(\alpha,\varepsilon)$
can be rephrased in the following manner which will be used later:
$$
N_{\bm{a}}(\alpha,\varepsilon)=\{n 
 \in \mathrm{ran}(\bm{a})
: \|n\alpha\|\ge \varepsilon\}
.$$

Also, we say that a sequence $\bm{a} \in \mathbb{Z}^\omega$ is \emph{thick} if $\mathrm{ran}(\bm{a})$ contains arbitrarily large intervals of integers (in particular, $\mathrm{ran}(\bm{a})$ is an infinite subset of $\mathbb{Z}$). 

If $\cA$ is a family of subsets of $\omega$ then the ideal generated by $\cA$ is the family of all $S\subseteq\omega$ such that
$$
    \exists_{k \in \omega} \exists_{A_0,\ldots,A_{k-1} \in \cA} 
    \exists_{F \in \mathrm{Fin}}\qquad
    S\subseteq A_0\cup\ldots\cup A_{k-1}\cup F.
$$
It is clear that the ideal generated by $\cA$ is an ideal: contains $\mathrm{Fin}$ and is stable under subsets and finite unions. However, it may happen that it is not proper.

\begin{definition}\label{def:IHa}
    For each thick sequence $\bm{a} \in \mathbb{Z}^\omega$ and for each subgroup $H$ of $\mathbb{T}$, let 
    $
    \mathcal{I}_{H,\bm{a}}
    $ 
    be the ideal on $\mathrm{ran}(\bm{a})$ generated by the family of sets $N_{\bm{a}}(\alpha,\varepsilon)$ with $\alpha \in H$ and $\varepsilon>0$. 
    Accordingly, define 
    $$
    \mathcal{J}_{H,\bm{a}}:=\{S\subseteq \omega: \{a_n: n \in S\} \in  \mathcal{I}_{H,\bm{a}}\}.
    $$
    In the special case of the sequence $\bm{a}=\bm{u}$, we simply write $\mathcal{J}_H:=\mathcal{I}_H:=\mathcal{I}_{H,\bm{u}}$. 
\end{definition}

It is worth noting that $\mathcal{I}_{H,\bm{a}}$ can be written as 
the family of all $S\subseteq \mathrm{ran}(\bm{a})$ for which there exist $k \in \omega$, $\alpha_0,\ldots,\alpha_{k-1} \in H$, $\varepsilon>0$, and a finite subset $F\subseteq \mathrm{ran}(\bm{a})$ such that $S\subseteq F\cup \bigcup_{i \in k}N_{\bm{a}}(\alpha_i,\varepsilon)$. Equivalently, it is the family of all $S\subseteq \mathrm{ran}(\bm{a})$ such that 
$$
\exists_{k \in \omega} 
\exists_{\alpha_0,\ldots,\alpha_{k-1} \in H} 
\exists_{m \in \omega}  
\exists_{N \in \omega} 
\forall_{n \in S\setminus [-N,N]} 
\exists_{i\in k}
\qquad 
\|n\alpha_i\|\ge 2^{-m}.
$$
\begin{remark}
We prove that each $\mathcal{I}_{H,\bm{a}}$ is a proper ideal,  that is $\mathrm{ran}(\bm{a})\notin \mathcal{I}_{H,\bm{a}}$. In fact, this is equivalent to
$$
    \forall_{k \in \omega} \forall_{\alpha_0,\ldots,\alpha_{k-1} \in \mathbb{T}} 
    \forall_{m \in \omega}  
    \forall_{N \in \omega}  
    \exists_{n \in \mathrm{ran}(\bm{a})\setminus [-N,N]} 
    \forall_{i \in k}
    \qquad \|\,n\alpha_i\,\|<2^{-m},
$$
    which holds by Dirichlet's approximation theorem (taking into account that $\mathrm{ran}(\bm{a})$ contains arbitrarily large intervals), cf. also Theorem \ref{thm:hardapproximation} below for a stronger claim. It follows that also $\mathcal{J}_{H,\bm{a}}$ is a (proper) ideal on $\omega$. 
\end{remark}

We list below some basic properties of $\mathcal{I}_{H,\bm{a}}$. 
If $\I$ is an ideal and $A\in\I^+$ then $\I\restriction A=\{B\cap A:B\in\I\}$. 

\begin{lemma}\label{lem:basicIHa}
Let $H$ be a subgroup of $\mathbb{T}$. Pick also thick sequences $\bm{a}, \bm{b} \in \mathbb{Z}^\omega$. 
Then the following  hold\text{:}
\begin{enumerate}[label={\rm (\roman*)}]
\item \label{item:1basiclemma} If $\mathrm{ran}(\bm{a})\subseteq \mathrm{ran}(\bm{b})$ then $\mathcal{I}_{H,\bm{a}}=\mathcal{I}_{H,\bm{b}}\upharpoonright \mathrm{ran}(\bm{a})$\textup{;}
\item \label{item:2basiclemma} 
$\mathcal{I}_{H,\bm{a}}=
\bigcup\{\mathcal{I}_{F,\bm{a}}:F\le H, F \text{ is finitely generated}\}$\textup{;}
\item \label{item:3basiclemma} If $H$ is countable then $\mathcal{I}_{H,\bm{a}}$ is countably generated\textup{.}
\end{enumerate}
\end{lemma}

Now, let $\bm{a}$ be a thick sequence of integers, and set
$A:=\mathrm{ran}(\bm a)$. By an abuse of notation,
if $\I$ is an ideal on the countably infinite set $A$, let us write 
$$\mathsf H_A(\mathcal I)
:=\{x\in\mathbb T:\forall\varepsilon>0\ 
\{n\in A:\|nx\|\geq\varepsilon\}\in\mathcal I\}.$$ 
As it follows, by the definition of $\mathcal{I}_{H,\bm{a}}$, if $H$ is a subgroup of $\mathbb{T}$ and $\mathcal{I}$ is an ideal on $A$ such that $H=\mathsf{H}_{A}(\mathcal{I})$ then necessarily $\mathcal{I}_{H,\bm{a}}\subseteq \mathcal{I}$. Essentially, the following theorem shows that $\mathcal{I}_{H,\bm{a}}$ is indeed the minimal ideal $\mathcal{I}$ for which such equality holds, cf. \eqref{eq:equalityrepresentations} below. 
(A proof of the next result in the case where $(a_n: n \in \omega)$ is an enumeration of $\mathbb{Z}$ can be found in \cite[Theorem 2.1]{MR2227021}.)

\begin{theorem}\label{thm:beigblock_representation}
Let $\bm{a}\in \mathbb{Z}^\omega$ be a thick sequence of integers. Then 
$$
H=\mathsf{H}_{\bm{a}}(\mathcal{J}_{H,\bm{a}})
$$
for each subgroup $H$ of $\mathbb{T}$.
\end{theorem}

In particular, since $\mathrm{ran}(\bm{u})=\omega$, it follows by Theorem \ref{thm:beigblock_representation} that for each subgroup $H$ of $\mathbb{T}$ there exists an ideal $\mathcal{I}$ on $\omega$ such that $H$ is $\mathcal{I}$-characterized (and, more precisely, $H=\mathsf{H}_{\bm{u}}(\mathcal{I})$ with $\mathcal{I}=\mathcal{I}_H$). 
This provides an answer to \cite[Question 6.6(a) and Question 6.6(b)]{protasov}. 
In addition, if $H$ and $K$ are two distinct subgroups of $\mathbb{T}$, then $\mathcal{J}_{H,\bm{a}} \neq \mathcal{J}_{K,\bm{a}}$, hence also  $\mathcal{I}_{H,\bm{a}} \neq \mathcal{I}_{K,\bm{a}}$.

At this point it is natural to ask which ideals are of the form $\mathcal{I}_H$. The following result gives a partial answer. 
Here, $\langle x\rangle$ stands for the subgroup generated by $x\in \mathbb{T}$. 
Two ideals $\mathcal{I}_1$ and $\mathcal{I}_2$ on countably infinite sets $W_1$ and $W_2$, respectively, are isomorphic if there exists a bijection $f: W_1\to W_2$ such that $f[S] \in \mathcal{I}_2$ if and only if $S \in \mathcal{I}_1$ for all $S\subseteq W_1$. $\Fin\oplus\mathcal{P}(\omega)$ is the ideal on $\{0,1\}\times\omega$ generated by the one-element family $\{\{1\}\times\omega\}$. $\Fin\otimes \emptyset$ is the ideal on $\omega^2=\omega\times \omega$ generated by the family $\{\{n\}\times\omega: n\in\omega\}$. 

\begin{lemma}\label{lem:basicIH}
Let $H$ be a subgroup of $\mathbb{T}$. Pick also a point $x \in \mathbb{T}$. 
Then the following  hold\text{:}
\begin{enumerate}[label={\rm (\roman*)}]
\item \label{item:4basiclemma} If $H\neq \{0\}$ then $\mathcal{I}_{H}$ is generated by sets with positive asymptotic density\textup{;}
\item \label{item:5basiclemma} If $x=0$, then $\I_{\langle x\rangle}=\Fin$\textup{;}
\item \label{item:6basiclemma} If $x\neq 0$ is rational, then $\I_{\langle x\rangle}$ is isomorphic to $\Fin\oplus\mathcal{P}(\omega)$\textup{;}
\item \label{item:7basiclemma} If $x$ is irrational, then $\I_{\langle x\rangle}$ is isomorphic to $\Fin\otimes \emptyset$\textup{.}
\end{enumerate}
\end{lemma}

In the following result, we provide a relationship between the complexity of $H$ and that of $\J_{H,\bm{a}}$. In addition, we characterize the $P$-property of $\J_{H,\bm{a}}$. Lastly, we show that if $H$ is countable then $\J_{H,\bm{a}}$ is a $P^+$-ideal, namely, for all decreasing sequences $(A_n: n\in \omega)$ in $\J_{H,\bm{a}}^+$ there exists $A \in \J_{H,\bm{a}}^+$ such that $A\setminus A_n$ is finite for all $n \in \omega$. 
\begin{proposition}\label{prop:HFsigmaIHFsigma}
Let $\bm{a}$ be a thick sequence, and let $H$ be a subgroup of $\mathbb{T}$. Then the following hold\text{:}
\begin{enumerate}[label={\rm (\roman*)}]
 \item \label{item:1complexIH} If $H$ is $F_\sigma$ then $\mathcal{J}_{H,\bm{a}}$ is an $F_\sigma$ ideal \textup{(}and, hence, a $P^+$-ideal\textup{)}\textup{;}
    \item \label{item:2complexIH} If $H$ is analytic then $\mathcal{J}_{H,\bm{a}}$ is an analytic ideal\textup{;}
    \item \label{item:4complexIH} $H$ is finite if and only if $\mathcal{I}_{H,\bm{a}}$ is a $P$-ideal\textup{;}
    \item \label{item:5complexIH} if $\bm{a}$ is finite-to-one, then $H$ is finite if and only if $\mathcal{J}_{H,\bm{a}}$ is a $P$-ideal\textup{.}
 \end{enumerate}
\end{proposition}


Theorem \ref{thm:beigblock_representation} also leads us to the following open question: 
\begin{question}\label{question:Istarrepresentation}
Is it true that every subgroup of $\mathbb{T}$ 
is $\I^\star$-characterized for some ideal $\I$?
\end{question}

Following \cite{MR1396895}, we recall that a subgroup $H$ of $\mathbb{T}$ is said to be \emph{Polishable} if it admits a Polish group topology $\tau$ on $H$ having the same Borel sets as $H$ when considered as a topological subgroup of $\mathbb{T}$, cf. \cite[Definition 1.18]{MR3461178}. 
Such topology $\tau$, if it exists, is necessarily unique \cite{MR1720580}. 
All Polishable subgroups are Borel, and they might attain arbitrarily high Borel complexity; detailed studies on Polishable subgroups can be found in \cite{MR2189217, MR2267154, MR5045103}. 
It is known that characterized subgroups are necessarily Polishable, see e.g. \cite[Corollary~1]{MR2729343} or \cite[Theorem 1.21]{MR3461178} and cf. Remark \ref{rmk:polishabilityZ} below. 
\begin{corollary}\label{cor:notPolishable}
    There exists an $F_\sigma$ ideal $\mathcal{I}$ on $\omega$ and an $F_\sigma$ subgroup $H$ of $\mathbb{T}$ such that $H$ is $\mathcal{I}$-characterized and not Polishable \textup{(}hence, not characterized\textup{)}.
\end{corollary}

\begin{remark}\label{rmk:polishabilityZ}
    It has been claimed in \cite[Proposition 2.4]{MR4078214} that the subgroup $\mathsf{H}_{\bm{a}}(\mathcal{Z})$ is Polishable for every sequence of natural numbers $\bm{a} \in \omega^\omega$. However, the argument given in \cite[Proposition 2.4]{MR4078214} appears to show that $\mathsf{H}_{\bm{a}}(\mathcal{Z})$, as a Borel subset of $\mathbb{T}$, admits a finer Polish topology. It does not seem to verify that this topology makes the group operations continuous, which is required for Polishability as a subgroup. 
    The same gap appears in \cite[Remark 1.12]{MR4780094}, replacing $\mathcal{Z}$ with an arbitrary analytic $P$-ideal. 
\end{remark}
Hence, we state the following open question.
\begin{question}\label{eq:polishability}
Is it true that, if $\mathcal{I}$ is an analytic $P$-ideal, every $\mathcal{I}$-characterized subgroup of $\mathbb{T}$ is Polishable?
\end{question}

\subsection{When the $\I$-characterized subgroup is the whole group $\mathbb{T}$?}

Next, we recover, with a different proof, the known result that only finitely supported sequences $\bm{a}$ generate the whole group $\mathbb{T}$, see \cite{Schoen}; cf. \cite[Theorem D]{MR4932230} and also \cite[Theorem 7.8]{MR419394} for a textbook exposition.
\begin{proposition}\label{prop:basiccase}
For each $\bm{a} \in \mathbb{Z}^\omega$, we have $\mathsf{H}_{\bm{a}}(\mathrm{Fin})=\mathbb{T}$ if and only if $\mathrm{supp}(\bm{a}) \in \mathrm{Fin}$. 
\end{proposition}

This leads us to the following question.
\begin{question}\label{question:support}
    For which ideals $\mathcal{I}$ on $\omega$ we have  
    \begin{equation}\label{eq:equivalencetorus}
    \mathsf{H}_{\bm{a}}(\mathcal{I})=\mathbb{T}
    \,\,\,\text{ if and only if }\,\,\,
    \mathrm{supp}(\bm{a}) \in \mathcal{I}
    \end{equation}
    for all integer sequences $\bm{a} \in \mathbb{Z}^\omega$? 
    In particular, does it hold for the ideal $\mathcal{Z}$? 
\end{question}

The first part of Question \ref{question:support} appeared also as \cite[Section 3, Question (i)]{MR4078214}; there, the authors claim that the study of Question \ref{question:support} is relatively easy, since using Proposition \ref{prop:basiccase} one can conclude the equivalence \eqref{eq:equivalencetorus} holds for the ideal $\mathcal{Z}$. However, we believe that Question \ref{question:support} is more delicate than what it seems at first. 

\begin{remark}\label{rmk:negative}
We remark that the equivalence \eqref{eq:equivalencetorus} does \emph{not} hold for all ideals $\mathcal{I}$ on $\omega$, even if we restrict to ideals with low topological complexities. In fact, choosing the thick sequence $\bm{u}$, it follows by Theorem \ref{thm:beigblock_representation} and Proposition \ref{prop:HFsigmaIHFsigma} that $\mathcal{J}_{\mathbb{T}}$ is an $F_\sigma$ ideal on $\omega$, $\mathsf{H}_{\bm{u}}(\mathcal{J}_{\mathbb{T}})=\mathbb{T}$, and $\mathrm{supp}(\bm{u})=\omega\setminus \{0\} \notin \mathcal{J}_{\mathbb{T}}$. In turn, this provides a negative answer to \cite[Question 7.22]{MR4932230}. 
\end{remark}

In the following result, we obtain a necessary technical condition on the integer sequence $\bm{a}$ to satisfy $\mathsf{H}_{\bm{a}}(\mathcal{I})=\mathbb{T}$.

\begin{proposition}\label{prop:necessaryfullT}
Let $\mathcal I$ be an ideal on $\omega$, and pick a sequence 
$\bm a=(a_n:n\in\omega)\in\mathbb Z^\omega$ such that $\mathrm{supp}(\bm{a}) \in \mathcal{I}^+$ and $\mathsf H_{\bm a}(\mathcal I)=\mathbb T$. 
Then, for each $S \in \mathcal{I}^+$ with $S\subseteq \mathrm{supp}(\bm{a})$, we have 
\begin{equation}\label{eq:hadamard}
\liminf_{j\to \infty}\frac{q_{j+1}}{q_j}=1,
\end{equation}
where $(q_j: j \in \omega)$ is the increasing enumeration of $\{|a_n|: n \in S\}$. 

\end{proposition}

The above result provides a generalization of \cite[Theorem 5.16]{MR4932230}, which corresponds to the case where $\bm{a}$ is a strictly increasing sequence of positive integers, $S=\omega$, and the inferior limit in \eqref{eq:hadamard} is replaced by an infimum. 

Below, we will use another preorder between ideals.  
If $\mathcal I$ and $\mathcal J$ are ideals on countably infinite sets $X$ and $Y$, respectively, we say that $\mathcal{I}$ is below $\mathcal{J}$ in the \emph{Katětov order}, shortened as $\mathcal I\leq_{\mathrm{K}} \mathcal J$, if there exists a map $f:Y\to X$ such that $f^{-1}[A]\in\mathcal J$ for every
$A\in\mathcal I$. 
Hence, $f[B]\in\mathcal I^+$ for all $B \in \mathcal{J}^+$.

Let us also recall the definition of the ideal $\mathcal{ED}$, which is known to be useful in  the study of the Katětov order on Borel ideals \cite{MR3600759, MR4247792, MR3696069}. To this aim, for $A\subseteq\omega^2$ and $i\in\omega$, write $A_i:=\{j\in\omega:(i,j)\in A\}$. Then
$$
\mathcal{ED}:=
\left\{A\subseteq\omega^2:\sup_{i>k}|A_i|<\infty \text{ for some }k\in\omega\right\}.
$$
$\mathcal{ED}$ is critical for a selectivity property of ideals: for every ideal $\I$ on a countably infinite set $X$, $\mathcal{ED}\leq_{\mathrm{K}}\I$ is equivalent to the existence of a partition of $X$ into sets in $\I$ such that every selector of that partition is in $\I$ \cite[p. 57]{alcantara-phd-thesis}.

We will be interested in the following property of an ideal $\I$:
\begin{equation}\label{eq:technicalP-EDfin}
\mathcal{ED}\not\leq_{\mathrm{K}}\I\restriction A 
\,\,\,\text{ for all }A \in \I^+.
\end{equation}
By Hru\v{s}\'{a}k's Category Dichotomy \cite[Theorem 3.1]{MR3696069}, in the case of Borel ideals $\I$, (\ref{eq:technicalP-EDfin}) implies $\I\leq_{\mathrm{K}}\mathsf{nwd}$ (the ideal $\mathsf{nwd}$ is defined above Corollary \ref{cor:1P^-B}).

Accordingly, we provide a sufficient condition on $\I$ which satisfies the property stated in Question \ref{question:support}. 
\begin{theorem}\label{thm:1P^-}
Let $\I$ be an ideal on $\omega$ satisfying (\ref{eq:technicalP-EDfin}). Then for every $\bm{a} \in \mathbb{Z}^\omega$ we have 
$\mathsf{H}_{\bm{a}}(\mathcal{I})=\mathbb{T}$ if and only if $\mathrm{supp}(\bm{a}) \in \mathcal{I}$. 
\end{theorem}

Next, we provide some practical instances where the hypotheses of Theorem \ref{thm:1P^-} are satisfied. 
Recall that an ideal $\mathcal I$ is \emph{tall} if every $A \in \mathrm{Fin}^+$ contains an infinite subset in $\mathcal{I}$. In the opposite direction, $\mathcal{I}$ is said to be \emph{nowhere tall} if for every $A\in\mathcal I^+$ there is an infinite set $B\subseteq A$ such that $\mathcal I\restriction B=[B]^{<\omega}$. Examples of nowhere tall ideals include countably generated ideals and ideals of the form $\{A\subseteq \omega: A\cap B \in \mathrm{Fin} \text{ for all $B \in \mathcal{B}$}\}$ where $\mathcal{B}$ is an almost disjoint family on $\omega$ (i.e.~$\cB$ consists of infinite subsets of $\omega$ and $B\cap C$ is finite for distinct $B,C\in \cB$); a characterization of nowhere tall ideals can be found in \cite[Theorem 5.3]{MR4041540}.

\begin{corollary}\label{cor:1P^-A} 
Let $\mathcal I$ be nowhere tall ideal. 
Then for every $\bm a\in\mathbb Z^\omega$ we have
$\mathsf H_{\bm a}(\mathcal I)=\mathbb T$ if and only if
$\operatorname{supp}(\bm a)\in\mathcal I$. 
\end{corollary}
In particular, taking into account that $\mathrm{Fin}$ is countably generated (hence, nowhere tall), Corollary \ref{cor:1P^-A} provides a generalization of Proposition \ref{prop:basiccase} above. 

To provide additional examples, set $D:=\mathbb Q\cap[0,1]$, and denote by $\mathsf{nwd}$ and $\mathsf{null}$ the ideals on $D$ defined respectively by
$$
\mathsf{nwd}:=\{A\subseteq D: A\text{ is nowhere dense}\}
\,\,\,\text{ and }\,\,\,
\mathsf{null}:=\{A\subseteq D:\lambda(\overline A)=0\},
$$
where the closure is taken in the usual topology of $[0,1]$ and $\lambda$ is the Lebesgue measure on $[0,1]$, see \cite{MR1955288}. It is easy to see that both $\mathsf{nwd}$ and $\mathsf{null}$ are not nowhere tall ideals (hence not covered by Corollary \ref{cor:1P^-A}). 
As usual, ideals on countably infinite sets are considered up to isomorphism, through a given bijection on $\omega$.
\begin{corollary}\label{cor:1P^-B} 
Suppose that $\mathcal I=\mathsf{nwd}$ or $\mathcal{I}=\mathsf{null}$.  
Then for every $\bm a\in\mathbb Z^\omega$ we have
$\mathsf H_{\bm a}(\mathcal I)=\mathbb T$ if and only if
$\operatorname{supp}(\bm a)\in\mathcal I$. 
\end{corollary}

\subsection{When there is a $\J$-characterized subgroup, which is not $\I$-characterized?}

In our last main result, we will use the sequence 
$$
\bm{b}=(b_n: n \in \omega) 
\quad \text{ defined by }\quad 
b_n:=2^n \text{ for all }n \in \omega. 
$$
We are going to provide sufficient conditions on ideals $\mathcal{I}$ and $\mathcal{J}$ for which $\mathscr{H}(\J)$ is not contained in $\mathscr{H}(\I)$, cf. also Remark \ref{rmk:notinclusion} below. 
We say that an ideal $\J$ on $\omega$ is translation invariant if for all $A\in\J$ and $n\in\omega$ we have $A-n:=\{k\in\omega:k+n\in A\}\in\J$.

\begin{theorem}\label{thm:2P^-}
Let $\I$ be an ideal which satisfies \eqref{eq:technicalP-EDfin}. Let also $\mathcal{J}$ be a translation invariant ideal, and suppose that 
\begin{equation}\label{eq:secondtechnicalcondition}
\J\not\leq_{\mathrm{K}}\I\restriction A \text{ for all }A\in\I^+.
\end{equation}
Then $\mathsf{H}_{\bm{b}}(\J)\notin \mathscr{H}(\I)$.
\end{theorem}

Some consequences are immediate. 
The first one is a generalization of \cite[Proposition 3.2]{MR4078214} and \cite[Theorem 2.7]{MR4436350} (in case of the sequence $\bm{b}=(2^n: n \in \omega)$).
\begin{corollary}\label{cor:Hpowernotfin}
Let $\mathcal{J}$ be a tall translation invariant ideal on $\omega$ \textup{(}for instance, $\mathcal{Z}$ or the summable ideal $\mathcal{I}_{1/n}:=\{A\subseteq\omega:\sum_{n\in A}\frac{1}{n+1}<\infty\}$\textup{)}. Then $\mathsf{H}_{\bm{b}}(\J)\notin \mathscr{H}(\mathrm{Fin})$.
\end{corollary}

On a similar note, in the case $\mathcal{J}=\mathcal{Z}$, we have also the following:
\begin{corollary}\label{cor:caseZ}
    $\mathsf{H}_{\bm{b}}(\mathcal{Z})\notin \mathscr{H}(\mathsf{nwd})$ and $\mathsf{H}_{\bm{b}}(\mathcal{Z})\notin \mathscr{H}(\mathsf{null})$. 
\end{corollary} 
Observe that both statements are stronger than $\mathsf{H}_{\bm{b}}(\mathcal{Z})\notin \mathscr{H}(\mathrm{Fin})$: in fact, both $\mathsf{nwd}$ and $\mathsf{null}$ are $F_{\sigma\delta}$ ideals \cite{MR1955288}, hence $\mathscr{H}(\mathrm{Fin})$ is contained in both $\mathscr{H}(\mathsf{nwd})$ and $\mathscr{H}(\mathsf{null})$ by Corollary \ref{RB:corollary}. 

As a last application, we consider the case of meager ideals. 
\begin{corollary}\label{cor:meager}
    Let $\mathcal{J}$ be a tall translation invariant meager ideal. Then 
    $$
    \mathscr{H}(\mathrm{Fin})\subsetneq \mathscr{H}(\mathcal{J}).
    $$
\end{corollary}

The proofs of the main results are given in Section \ref{sec:proofs}.

\section{Auxiliary Results}\label{sec:preliminaries}

Here, 
we list some preliminary results which will be needed to prove our main results.
\begin{proposition}\label{prop:interpolation}
Let $(a_n: n \in \omega)$ be a sequence of positive integers such that $\lim_n a_{n+1}/a_n=\infty$. Then for every sequence $(t_n:n\in\omega)$ in $\mathbb T$, there exists $x\in\mathbb T$ such that 
    $$
    \lim_{n\to \infty} \|a_nx-t_n\|=0. 
    $$
\end{proposition}
\begin{proof} 
Define $b_{n}:=a_{n+1}/a_n$ and $\varepsilon_n:=b_n^{-1/2}$ for all $n \in \omega$, and observe that $\lim_n \varepsilon_n=0$ and $\lim_n \varepsilon_n b_n=\infty$.  
Pick $n_0 \in\omega$ such that $\varepsilon_n<1/2$ for all $n\ge n_0$, and define 
$$
E_n:=\{z\in\mathbb T:\|a_nz-t_n\|\le \varepsilon_n\}
$$
for all $n \in \omega$. 
It follows that, for all $n\ge n_0$, $E_n$ is
a union of $a_n$ closed arcs in $\mathbb T$, each of length $L_n:=2\varepsilon_n/a_n$. 
The centers of the components of $E_{n+1}$ are equally spaced with spacing
$d_{n+1}:=1/a_{n+1}$, and each component of $E_{n+1}$ has radius
$r_{n+1}:=\varepsilon_{n+1}/a_{n+1}$.

At this point, observe that a component of $E_{n+1}$ is contained in some component $J$ of $E_n$ provided that 
one of the centers of $E_{n+1}$ lies at distance at least
$r_{n+1}$ from both endpoints of $J$.
Such a center exists whenever
$$
    L_n-2r_{n+1}
    >
    d_{n+1}.
$$
The latter inequality can be rewritten equivalently as
$$
    2\varepsilon_nb_n
    >
    1+2\varepsilon_{n+1}.
$$
Since this inequality holds for all sufficiently large $n$, it is possible to pick an integer $n_1\ge n_0$ such that, for each $n\ge n_1$, every component of $E_n$ contains some component of $E_{n+1}$. Hence there exists a decreasing sequence $(J_n: n\ge n_1)$ of nonempty compact sets such that $J_n$ is a component of $E_n$ for each $n\ge n_1$. 

By compactness of $\mathbb T$, it is possible to pick $x \in \bigcap_{n\ge n_1}J_n$. Since $J_n\subseteq E_n$ for every $n\ge n_1$, we have $\|a_nx-t_n\|\le \varepsilon_n$ for all $n\ge n_1$. Since $\lim_n \varepsilon_n=0$, we conclude that $\lim_n \|a_nx-t_n\|=0$. 
\end{proof}

\medskip

\begin{lemma}\label{lem:obviousinclusion}
    Pick $\bm{a}, \bm{b} \in \mathbb{Z}^\omega$ such that $\bm{b}$ is a subsequence of $\bm{a}$. Then $\mathsf{H}_{\bm{a}}(\mathrm{Fin})\subseteq \mathsf{H}_{\bm{b}}(\mathrm{Fin})$. 
\end{lemma}
\begin{proof}
    This is clear by the fact that every subsequence of a convergent sequence is convergent to the same limit. 
\end{proof}

\begin{lemma}\label{lem:almostobvious}
    Let $\bm{a}\in \mathbb{Z}^\omega$ be a bounded sequence with $\mathrm{supp}(\bm{a})$ infinite.   
    Then $\mathsf{H}_{\bm{a}}(\mathrm{Fin})$ is finite. 
\end{lemma}
\begin{proof}
    Set $k:=\sup_n|a_n|$. Then there exists an integer $m\neq 0$ such that $|m|\le k$ and $a_n=m$ for infinitely many $n \in \omega$. Observe that if $\lim_n \|a_nx\|=0$ then necessarily $\|mx\|=0$. Hence 
    $
    \mathsf{H}_{\bm{a}}(\mathrm{Fin})\subseteq 
    \{x \in \mathbb{T}: \|mx\|=0\},
    $ 
    which is a finite set.
\end{proof}

\medskip

Let us recall now a special case of Kronecker's simultaneous approximation theorem.
\begin{theorem}\label{thm:Acasselkronecker}
    Pick $\theta_1,\ldots,\theta_d,\alpha_1,\ldots,\alpha_d\in\mathbb R$ and suppose that, 
    for every $u_1,\ldots,u_d\in\mathbb Z$, we have $u_1\alpha_1+\cdots+u_d\alpha_d\in\mathbb Z$ whenever $u_1\theta_1+\cdots+u_d\theta_d\in\mathbb Z$. 
Then for every $\varepsilon>0$ there exists $N\in\mathbb Z$ such that $\|N\theta_j-\alpha_j\|<\varepsilon$ for every $j=1,\ldots,d$.
\end{theorem}
\begin{proof}
    See \cite[Chapter III, Theorem IV]{MR87708} in the case $m=1$. 
\end{proof}

Hereafter, given a subset $Y\subseteq \mathbb{T}$, we write $\langle Y\rangle$ for the smallest subgroup of $\mathbb{T}$ containing $Y$. Also, given (not necessarily distinct) $y_1,\ldots,y_k \in \mathbb{T}$, we write  $\langle y_1,\ldots,y_k\rangle:=\langle \{y_1,\ldots,y_k\}\rangle$. 

\begin{corollary}\label{cor:kronecker}
Pick $y_1,\ldots,y_d\in\mathbb T$ and 
let $\psi:\langle y_1,\ldots,y_d\rangle\to\mathbb T$ be a group homomorphism. Then, for every $\varepsilon>0$, there exists $N\in\mathbb Z$ such that $\|Ny_j-\psi(y_j)\|<\varepsilon$ for every $j=1,\ldots,d$.
\end{corollary}
\begin{proof}
    Pick reals $\theta_1,\ldots,\theta_d,\alpha_1,\ldots,\alpha_d\in\mathbb R$ such that $y_j=\theta_j+\mathbb Z$ and $\psi(y_j)=\alpha_j+\mathbb Z$ for every $j=1,\ldots,d$.
    Now, suppose that $u_1,\ldots,u_d\in\mathbb Z$ satisfies  $u_1\theta_1+\cdots+u_d\theta_d\in\mathbb Z$, thus $u_1y_1+\cdots+u_dy_d=0$ in $\mathbb T$. Since $\psi$ is a homomorphism, it follows that
\[
u_1\psi(y_1)+\cdots+u_d\psi(y_d)=\psi(u_1y_1+\cdots+u_dy_d)=\psi(0)=0.
\]
Equivalently, $u_1\alpha_1+\cdots+u_d\alpha_d\in\mathbb Z$. 
Thanks to Theorem \ref{thm:Acasselkronecker}, for every $\varepsilon>0$ there exists $N\in\mathbb Z$ such that $\|N\theta_j-\alpha_j\|<\varepsilon$ for every $j=1,\ldots,d$. In turn, this is equivalent to $\|Ny_j-\psi(y_j)\|<\varepsilon$ for every $j=1,\ldots,d$.
\end{proof}

Next, let us recall that a set of integers $A\subseteq \mathbb{Z}$ is said to be \emph{thick} if it contains arbitrarily large intervals of integers (hence, $\bm{a}$ is thick if $\mathrm{ran}(\bm{a})$ is thick). An alternative proof of the following result in the case $A=\mathbb{Z}$ can be obtained from \cite[Theorem 2.1]{MR2227021}.  
\begin{theorem}\label{thm:hardapproximation}
Pick $h_1,h_2,\ldots,h_k,x \in \mathbb{T}$ with $x\notin \langle h_1,\ldots,h_k\rangle$. Let $A\subseteq\mathbb Z$ be a thick set of integers. Then, for every $\varepsilon>0$, there exists $n\in A$ such that
$$
\|nh_i\|<\varepsilon 
\quad\text{for every }i=1,\ldots,k
\quad \text{ and }\quad 
\|nx\|>\nicefrac14.
$$
\end{theorem}

\begin{proof}
Put $H:=\langle h_1,\ldots,h_k\rangle$ and $L:=\langle h_1,\ldots,h_k,x\rangle$. Every element of $L$ has the form $h+sx$, with $h\in H$ and $s\in\mathbb Z$. Define 
$$
R:=\{s\in\mathbb Z:sx\in H\}.
$$
Note that $R$ is a subgroup of $\mathbb Z$. In addition, since $x\notin H$, we have $1\notin R$. Therefore either $R=\{0\}$, or $R=r\mathbb Z$ for some integer $r\geq 2$.

\begin{claim}\label{claim:homom}
    There exists an homomorphism $\varphi:L\to\mathbb T$ such that 
\begin{equation}\label{eq:claimvarphi}
\varphi[H]=\{0\}
\quad \text{ and }\quad
\|\varphi(x)\|\geq \nicefrac{1}{3}.
\end{equation}
\end{claim}

\begin{proof}
First, suppose that $R=\{0\}$. If $h+sx=h^\prime+s^\prime x$, with $h,h^\prime\in H$ and $s,s^\prime\in\mathbb Z$, then $(s-s^\prime)x=h^\prime -h\in H$, hence $s-s^\prime \in R$, and so $s=s^\prime$. Thus the integer $s$ is uniquely determined by the element $h+sx$. Define $\varphi(h+sx):=s/3+\mathbb Z$. Then it is routine to observe that $\varphi$ is a well-defined homomorphism $L\to \mathbb{T}$ which satisfies \eqref{eq:claimvarphi}.

Now suppose that $R=r\mathbb Z$ for some $r\geq 2$. Set $j:=\lfloor \nicefrac{r}{2}\rfloor$, so that $\|\nicefrac{j}{r}+\mathbb Z\|\geq \nicefrac{1}{3}$. 
At this point, define $\varphi(h+sx):=\nicefrac{sj}{r}+\mathbb Z$. This is well-defined. In fact, if $h+sx=h^\prime+s^\prime x$, then $(s-s^\prime)x=h^\prime-h\in H$, so $s-s^\prime\in R=r\mathbb Z$. Thus $s-s^\prime=rq$ for some $q\in\mathbb Z$, and consequently 
$\nicefrac{sj}{r}-\nicefrac{s^\prime j}{r}=qj\in\mathbb Z$. 
Hence $\nicefrac{sj}{r}+\mathbb Z=\nicefrac{s^\prime j}{r}+\mathbb Z$. As in the previous case, it follows that $\varphi$ is a well-defined homomorphism $L\to \mathbb{T}$ which satisfies \eqref{eq:claimvarphi}. 
\end{proof}

Now, let $\varphi: L\to \mathbb{T}$ be as in Claim \ref{claim:homom}, and define $y:=(h_1,\ldots,h_k,x)\in\mathbb T^{k+1}$ and 
$$
a:=(\varphi(h_1),\ldots,\varphi(h_k),\varphi(x))=(0,\ldots,0,\varphi(x))\in\mathbb T^{k+1}.
$$
In addition, consider the compact set 
$$
C:=\overline{\{ny:n\in\mathbb Z\}}\subseteq\mathbb T^{k+1},
$$
where $ny:=(nh_1,\ldots,nh_k,nx)$. 
\begin{claim}\label{claim:closure}
$a\in C$.  
\end{claim}
\begin{proof}
It follows from Corollary \ref{cor:kronecker} applied to $y=(h_1,\ldots,h_k,x)$ and $\psi=\varphi$. 
\end{proof}

Fix $\varepsilon>0$ and define
$$
U:=\{z\in\mathbb T^{k+1}:\|z_i\|<\varepsilon\text{ for all }i=1,\ldots,k,\text{ and }\|z_{k+1}\|>\nicefrac{1}{4}\}.
$$
Observe that $a \in U$. Now, define $B:=\{n\in\mathbb Z: ny\in U\}$, so that $n\in B$ if and only if $\|nh_i\|<\varepsilon$ for every $i=1,\ldots,k$ and $\|nx\|>\nicefrac{1}{4}$. To complete the proof, it will be enough to show that 
$$
A\cap B\neq\emptyset.
$$

\begin{claim}\label{claim:Bsyndetic}
    There is a finite set $F\subseteq\mathbb Z$ such that $\mathbb Z=B+F$ (that is, $B$ is syndetic).
\end{claim}
\begin{proof}
On $\mathbb T^{k+1}$ let us consider the supremum metric $d(z,z^\prime):=\max_i \|z_i-z_i^\prime\|$, and denote the open balls with center $z$ and radius $r>0$ by $B(z,r)$. 
Since $U$ is open and $a\in U$, it is possible to choose $\rho>0$ such that $B(a,2\rho)\subseteq U$. Since $a\in C$ by Claim \ref{claim:closure}, there is $p\in\mathbb Z$ such that $d(py,a)<\rho$. 

Now, let $V:=B(0,\rho)$ and $S:=\{r\in\mathbb Z:ry\in V\}$. 
Since $\{ny:n\in\mathbb Z\}$ is dense in $C$, the family $\{ny+V:n\in\mathbb Z\}$ is an open cover of $C$. By compactness, there are $q_1,\ldots,q_\ell\in\mathbb Z$ such that 
$
C\subseteq (q_1y+V)\cup\cdots\cup(q_\ell y+V).
$ 
Fix any $q\in\mathbb Z$. Since $qy\in C$, there is $j\in\{1,\ldots,\ell\}$ such that $qy\in q_jy+V$. Hence $(q-q_j)y\in V$, so $q-q_j\in S$, and therefore $q\in S+q_j$. This proves that $S$ is syndetic and, more precisely, 
$$
\mathbb Z=S+\{q_1,\ldots,q_\ell\}.
$$

If $s\in S$, then $d(sy,0)<\rho$. Since $d(py,a)<\rho$, we get $d((p+s)y,a)<2\rho$, and hence $(p+s)y\in U$. Therefore $p+s\in B$, and so $p+S\subseteq B$. Since $S$ is syndetic, then $B$ is syndetic as well. 
\end{proof}

To conclude, pick a finite set $F\subseteq\mathbb Z$ such that $\mathbb Z=B+F$ as in Claim \ref{claim:Bsyndetic}. 
Write $f_{\min}:=\min F$ and $f_{\max}:=\max F$. Since $A$ is thick, we can choose $m\in\mathbb Z$ such that $\{m,m+1,\ldots,m+f_{\max}-f_{\min}\}\subseteq A$. Since $m+f_{\max}\in\mathbb Z=B+F$, there are $b\in B$ and $f\in F$ such that $m+f_{\max}=b+f$, hence $b=m+f_{\max}-f$. Since $f_{\min}\leq f\leq f_{\max}$, we have $m\leq b\leq m+f_{\max}-f_{\min}$. It follows that $b\in A$. Therefore $b \in A\cap B$, which completes the proof. 
\end{proof}

\medskip

\begin{lemma}\label{lem:piBorel}
Let $\bm{a}$ be an integer sequence and define
$$
\pi:2^\omega\to 2^Z,\qquad \pi(S):=\{a_n:n\in S\}, \quad \text{where }
Z:=\operatorname{ran}(\bm a).
$$
Then $\pi$ is Borel. Moreover, if $C\subseteq 2^Z$ is hereditary compact,
then $\pi^{-1}[C]$ is compact. 
\end{lemma}

\begin{proof}
First we show that \(\pi\) is Borel. For \(z\in Z\), we have
$$
\pi^{-1}[\{A\subseteq Z:z\in A\}]
=
\{S\subseteq\omega:\exists_{n\in S}\ a_n=z\}.
$$
This set is open in \(2^\omega\).  
Similarly, $\pi^{-1}[\{A\subseteq Z:z\notin A\}]$ is closed in $2^\omega$. 
Hence the preimage under $\pi$ of each
basic clopen subset of $2^Z$ is Borel. Therefore $\pi$ is Borel.

Now, let $C\subseteq 2^Z$ be hereditary compact. We show that $\pi^{-1}[C]$ is closed. Let $S\notin\pi^{-1}[C]$, so $\pi(S)\notin C$. 
Since \(C\) is closed, there is a finite set
$F\subseteq\pi(S)$ such that $F\notin C$. 
Indeed, otherwise every finite subset of $\pi(S)$ would belong to \(C\), and
by compactness we would get $\pi(S)\in C$. For each \(z\in F\), choose \(n_z\in S\) such that \(a_{n_z}=z\). Then the
basic open neighbourhood 
$U:=\{T\subseteq\omega:n_z\in T\text{ for every }z\in F\}$ 
of \(S\) satisfies $\pi(T)\supseteq F$ for every \(T\in U\). 
Since \(C\) is hereditary and \(F\notin C\), no such
\(\pi(T)\) can belong to \(C\). Hence 
$U\cap \pi^{-1}[C]=\emptyset$.  
Thus \(\pi^{-1}[C]\) is closed, hence compact in $2^\omega$. 
\end{proof}

\medskip

\begin{lemma}\label{lem:charactePideal}
Let \(X\) and \(Y\) be countably infinite sets, let \(f:X\to Y\) be a
finite-to-one surjection, and let \(\mathcal I\) be an ideal on \(Y\).
Define an ideal \(\mathcal J\) on \(X\) by
\[
\mathcal J:=\{A\subseteq X:f[A]\in\mathcal I\}.
\]
Then \(\mathcal I\) is a \(P\)-ideal if and only if \(\mathcal J\) is a
\(P\)-ideal.
\end{lemma}

\begin{proof}
First, suppose that $\mathcal I$ is a $P$-ideal. Let
$(A_n:n\in\omega)$ be a sequence of sets in $\mathcal J$. Then $f[A_n]\in\mathcal I$ for every $n\in\omega$. Since $\mathcal I$ is a $P$-ideal, there is
$B\in\mathcal I$ such that $f[A_n]\setminus B$ is finite for every $n\in\omega$. 
Put $A:=f^{-1}[B]$. Then $f[A]=B\in\mathcal I$, so $A\in\mathcal J$. 
Moreover, for every $n\in\omega$, 
$$
A_n\setminus A
\subseteq f^{-1}[f[A_n]\setminus B].
$$
Since $f[A_n]\setminus B$ is finite and $f$ is finite-to-one, the set $f^{-1}[f[A_n]\setminus B]$ is finite. 
Therefore $A_n\setminus A$ is finite for every $n\in\omega$.
Thus $\mathcal J$ is a $P$-ideal.

Conversely, suppose that $\mathcal J$ is a $P$-ideal. Let
$(B_n:n\in\omega)$ be a sequence of sets in $\mathcal I$. For each $n\in\omega$, put $A_n:=f^{-1}[B_n]$.  
Then $f[A_n]=B_n\in\mathcal I$, 
so $A_n\in\mathcal J$. Since $\mathcal J$ is a $P$-ideal, there is $A\in\mathcal J$ such that $A_n\setminus A$ is finite for every $n\in\omega$. 
Put $B:=f[A]$. 
Then $B\in\mathcal I$, because $A\in\mathcal J$. We claim that $B_n\setminus B$ is finite for every $n\in\omega$. To this aim, fix $n\in\omega$. 
If $y\in B_n\setminus B$, 
then, since $f$ is onto, it is possible to choose $x\in X$ such that $f(x)=y$. 
Since $y\in B_n$, every such $x$ belongs to $f^{-1}[B_n]=A_n$. 
Since $y\notin B=f[A]$, no such $x$ belongs to $A$. Hence
$$
f^{-1}[\{y\}]\subseteq A_n\setminus A.
$$
In particular, choosing one point from each nonempty 
$f^{-1}[\{y\}]$, we get an injection from $B_n\setminus B$ into the finite set $A_n\setminus A$. Therefore $B_n\setminus B$ is finite. Hence $\mathcal I$ is a $P$-ideal. 
\end{proof}

\medskip

\begin{lemma}\label{lem:P^-}
Let $\I$ be an ideal on $\omega$ which satisfies \eqref{eq:technicalP-EDfin}, as in Theorem \ref{thm:1P^-}. 
Then for every $\bm{a}\in\mathbb{Z}^\omega$ such that $\I\text{-}\lim_{n}|a_n|=\infty$ there is a sequence $(k(n))_{n\in\omega} \in \omega^\omega$ such that 
$$
\{k(n):n\in\omega\}\in\I^+
\quad \text{ and }\quad 
\lim_{n\to \infty}\frac{a_{k(n+1)}}{a_{k(n)}}=\infty.
$$
\end{lemma}
\begin{proof}
Fix $\bm a\in\mathbb Z^\omega$ such that $\mathcal I\text{-}\lim_n |a_n|=\infty$. Taking into account that  $\mathrm{supp}(\bm{a})\in\mathcal I^+$ and replacing $\bm a$ by $-\bm a$, if necessary, 
we may assume that 
$$
E_1:=\{n\in\omega:a_n>0\}\in\mathcal I^+.
$$
For every $k\in\omega$, we have $\{n\in E_1:a_n<k\}\in\mathcal I$. Then $(\{n\in E_1:a_n=k\}: k\in\omega)$ is a partition of $E_1$ into sets belonging to $\I\restriction E_1$. Since $\mathcal{ED}\not\leq_{\mathrm{K}}\I\restriction E_1$, there is $E_2\subseteq E_1$ such that $E_2\in\mathcal I^+$ and $|\{n\in E_2:a_n=k\}|=1$ for all $k\in\omega$. Let $(m(n):n\in\omega)$ be an enumeration of $E_2$ such that $a_{m(n)}\leq a_{m(n+1)}$ for all $n\in\omega$.

Inductively choose a strictly increasing sequence $(b_n:n\in\omega)$ in $\omega$ such
that 
$$
b_0=0
\qquad \text{ and }\qquad 
\frac{a_{m(b_{n+1})}}{a_{m(b_n)}}>n+1 \,\,\text{ for all }n \in \omega. 
$$
Define $F_0:=\{m(b_0)\}$ and $F_{n+1}:=\{m(i):b_n<i\leq b_{n+1}\}$ for every $n\in\omega$. Then $(F_n:n\in\omega)$ is a partition of $E_2$ into finite sets, and $\bigcup_n F_n=E_2\in\mathcal I^+$. Since $\mathcal{ED}\not\leq_{\mathrm{K}} \mathcal I\restriction E_2$, there is $E_3\in\mathcal I^+$ such that
$$
|E_3\cap F_n|=1 
\,\,\,\text{ for all }n \in \omega. 
$$

Now either 
$E_3\cap\bigcup_{n}F_{2n}\in\mathcal I^+$ or $E_3\cap\bigcup_{n}F_{2n+1}\in\mathcal I^+$. 
Without loss of generality, assume that
$$
E_4:=E_3\cap\bigcup_{n\in\omega}F_{2n}\in\mathcal I^+.
$$
For every $n$, choose the unique $e(n)\in\omega$ such that $E_4\cap F_{2n}=\{m(e(n))\}$, and put $k(n):=m(e(n))$. 
Then $\{k(n):n\in\omega\}=E_4\in\mathcal I^+$. 
Finally, since $e(n)\leq b_{2n}$ and $e(n+1)>b_{2n+1}$, the monotonicity of
$(a_{m(i)}:i\in\omega)$ gives
$$
\frac{a_{k(n+1)}}{a_{k(n)}}
=
\frac{a_{m(e(n+1))}}{a_{m(e(n))}}
\geq
\frac{a_{m(b_{2n+1})}}{a_{m(b_{2n})}}
>
2n+1.
$$
We conclude that $\lim_n a_{k(n+1)}/a_{k(n)}=\infty$. 
\end{proof}

\section{Proofs of Main results}\label{sec:proofs}

\begin{proof}
[Proof of Theorem \ref{thm:Pproperty}] 
\ref{item:01Pidealchar} $\implies$ \ref{item:02Pidealchar}. 
It follows from the folklore fact that, if $\mathcal{I}$ is a $P$-ideal then $\mathcal{I}$-convergence and $\mathcal{I}^\star$-convergence coincide \cite[Prop.~3.2 and Thm.~3.2]{MR1844385}.

\medskip

\ref{item:02Pidealchar} $\implies$ \ref{item:03Pidealchar}. This is clear. 

\medskip

\ref{item:03Pidealchar} $\implies$ \ref{item:01Pidealchar}. 
Assume that $\mathcal I$ is not a $P$-ideal and let $\bm a \in (\omega\setminus \{0\})^\omega$ be a positive sequence such that $\lim_n a_n/a_{n+1}=0$. It will be enough to show that $\mathsf H_{\bm a}(\mathcal I)\neq \mathsf H^\star_{\bm a}(\mathcal I)$. 

Since $\mathcal I$ is not a $P$-ideal, there is a sequence  $(B_k:k\in\omega)$ of sets in $\mathcal I$ such that 
\begin{equation}\label{eq:contradicitonPideal}
\forall_{B \in \mathcal{I}}\, \exists_{k \in \omega} \qquad |B_k\setminus B|=\omega.
\end{equation}
Replacing, if necessary, each $B_k$ with $B_k\setminus \bigcup_{t \in k}B_t$, it can be assumed without loss of generality that the sets $B_k$ are pairwise disjoint. Observe that $B_\infty:=\bigcup_j B_j$ does not belong to $\mathcal{I}$. 
Now, pick a sequence $(\delta_k: k\in \omega)$ in $\mathbb{T}$ such that $\delta_k\neq 0$ for all $k \in \omega$ and $\lim_k \delta_k=0$. Then, define the sequence $(t_n:n\in\omega)$ by
\begin{displaymath}
t_n:=
    \begin{cases}
        \,\delta_k & \text{ if } n\in B_k \text{ for some }k \in \omega,\\
        \,0      & \text{ if } n\notin B_\infty,
    \end{cases}
\end{displaymath}
for all $n \in \omega$. 
It follows from construction that $\mathcal{I}\text{-}\lim_n t_n=0$. On the other hand, taking into account \eqref{eq:contradicitonPideal}, the sequence $(t_n: n \in \omega)$ is not $\mathcal{I}^\star$-convergent. 
To conclude, thanks to Proposition \ref{prop:interpolation}, there exists $x\in\mathbb T$ such that $\lim_n (a_nx-t_n)=0$. It follows from \cite[Lemma 3.3 and Lemma 3.5(i)]{MR3920799} that $\mathcal{I}\text{-}\lim_n a_nx=0$ and, with a similar reasoning, $(a_nx: n \in \omega)$ is not $\mathcal{I}^\star$-convergent. Therefore $x \in \mathsf H_{\bm a}(\mathcal I)\setminus \mathsf H^\star_{\bm a}(\mathcal I)$. 
\end{proof}

\medskip

\begin{proof}
[Proof of Theorem \ref{thm:upperboundscomplexity}]
First, observe that 
$$
\mathsf{H}_{\bm{a}}(\mathcal{I})=\bigcap_{k\in\omega}\left\{x\in\mathbb{T}: \{n\in\omega: \|a_nx\|> 2^{-k} \}\in\I\right\}=\bigcap_{k\in\omega}\varphi^{-1}_k[\I],
$$
where $\varphi_k:\mathbb{T}\to\mathcal{P}(\omega)$ is given by $\varphi_k(x):=\{n\in\omega: \|a_nx\|> 2^{-k} \}$ for each $k \in \omega$. 

\begin{claim}\label{claim:Borel1} Each $\varphi_k$ is of Borel class $1$, that is, $\varphi_k^{-1}[U]$ is $F_\sigma$ for every open $U\subseteq\mathcal{P}(\omega)$.
\end{claim} 
\begin{proof}
Fix any $m\in\omega$ and notice that
\[
\varphi^{-1}_k[\{A\subseteq\omega:m\notin A\}]=\{x\in\mathbb{T}:\|a_mx\|\leq 2^{-k}\}
\]
is closed in $\mathbb{T}$, while
\[
\varphi^{-1}_k[\{A\subseteq\omega:m\in A\}]=\{x\in\mathbb{T}:\|a_mx\|> 2^{-k}\}
\]
is open in $\mathbb{T}$. The claim follows from the fact that sets of the type $\{A\subseteq\omega:m\notin A\}$ and $\{A\subseteq\omega:m\in A\}$ constitute a subbase of the topology of $\mathcal{P}(\omega)$.
\end{proof}

\medskip

\ref{item:1upperbound}. Let $(C_i: i \in \omega)$ be a sequence of hereditary compact subsets of $\mathcal{P}(\omega)$ which satisfies \eqref{eq:Farahrepresentation}. Now, fix $i,k\in \omega$ define
$C_{i,n}:=\{A\subseteq\omega:A\setminus n\in C_i\}$ for each $n \in \omega$. 
Then every $C_{i,n}$ is also a hereditary compact subset of $\mathcal P(\omega)$, because the map $A\mapsto A\setminus n$ is continuous and $C_i$ is hereditary compact. 
It follows that there is a family $\mathcal B\subseteq\mathrm{Fin}$ such that
$C_{i,n}=\bigcap_{F\in\mathcal B}\{A\subseteq\omega:F\not\subseteq A\}$. Therefore
\[
\varphi_k^{-1}[C_{i,n}]
=
\bigcap_{F\in\mathcal B}\bigcup_{m\in F}
\{x\in\mathbb T:m\notin\varphi_k(x)\}.
\]
For every $m\in\omega$, the set $\{x\in\mathbb T:m\notin\varphi_k(x)\}$ is equal to
$\{x\in\mathbb T:\|a_mx\|\leq 2^{-k}\}$, hence it is closed. Since $F$ is finite, it follows that 
$\varphi_k^{-1}[C_{i,n}]$ is closed in $\mathbb{T}$. 

Using \eqref{eq:Farahrepresentation}, we conclude that 
$$
\mathsf H_{\bm a}(\I)
=\bigcap_{k\in\omega}\varphi_k^{-1}[\I]
=\bigcap_{k\in\omega}\bigcap_{i\in\omega}\bigcup_{n\in\omega}\varphi_k^{-1}[C_{i,n}]
$$
is an $F_{\sigma\delta}$ subgroup. 

\medskip

\ref{item:2upperbound}. Thanks to Claim \ref{claim:Borel1}, if $\mathcal{I}$ is a $\Pi^0_\alpha$ ideal then by transfinite induction $\varphi_k^{-1}[\I]$ is $\Pi^0_{\alpha+1}$ for each $k \in \omega$. Hence, it follows that $\mathsf H_{\bm a}(\I)
=\bigcap_{k}\varphi_k^{-1}[\I]$ is $\Pi^0_{\alpha+1}$ as well. 

\medskip

\ref{item:3upperbound}. It follows from the item \ref{item:2upperbound} since every $\Sigma^0_\alpha$ subset of $\mathbb{T}$ is $\Pi^0_{\alpha+1}$. 

\medskip

\ref{item:4upperbound}. It follows from the identity $\mathsf H_{\bm a}(\I)
=\bigcap_{k}\varphi_k^{-1}[\I]$, Claim \ref{claim:Borel1}, and the fact that the classes of Borel, analytic and co-analytic sets are stable under Borel preimages and countable intersections. 

\medskip

\ref{item:5upperbound}. We will show that
$\mathsf H^\star_{\bm a}(\I)$ is the projection of an analytic subset of
$\mathbb T\times\mathcal P(\omega)$. 
\begin{claim}\label{claim:RBorel}
$R:=\{(x,S)\in\mathbb T\times\mathcal P(\omega):
\forall_{k\in\omega}\ \varphi_k(x)\setminus S\in\mathrm{Fin}\}$ is Borel.
\end{claim}
\begin{proof}
It is enough to prove that, for every $k\in\omega$, the set 
$$
\{(x,S):\varphi_k(x)\setminus S\in\mathrm{Fin}\}
=
\bigcup_{N\in\omega}\bigcap_{n\geq N}
\left(\{(x,S):n\in S\}\cup\{(x,S):\|a_nx\|\leq 2^{-k}\}\right).
$$
is $F_\sigma$. Indeed, 
for each fixed $n\in\omega$, the set $\{(x,S):n\in S\}$ is clopen in
$\mathbb T\times\mathcal P(\omega)$, while
$\{(x,S):\|a_nx\|\leq 2^{-k}\}$ is closed. Thus, the set in the right-hand side above is
$F_\sigma$. Therefore $R$ is $F_{\sigma\delta}$, hence Borel.
\end{proof}
Now, observe that
$$
\mathsf H^\star_{\bm a}(\I)
=
\{x\in\mathbb T:\exists_{S\in\I}\ (x,S)\in R\}.
$$
In fact, $x\in\mathsf H^\star_{\bm a}(\I)$ if and only if there is $S\in\I$ such that
$\lim_{n\in\omega\setminus S}a_nx=0$, and this is equivalent to saying that
$\varphi_k(x)\setminus S$ is finite for every $k\in\omega$.

Since $\I$ is analytic, the set $\mathbb T\times\I$ is analytic in
$\mathbb T\times\mathcal P(\omega)$. It follows by Claim \ref{claim:RBorel} that 
$R\cap(\mathbb T\times\I)$ is analytic, and its projection on the first coordinate is
analytic as well. Therefore $\mathsf H^\star_{\bm a}(\I)$ is analytic.
\end{proof}

\medskip

\begin{proof}
[Proof of Theorem \ref{thm:RKinclusion}]
\ref{item:1RK} Pick a map $f:\omega\to\omega$ such that 
$f^{-1}[A]\in\J$ if and only if $A\in\I$ for all $A\subseteq\omega$. Fix any $\bm{a} \in \mathbb{Z}^\omega$ and define $\bm{b} \in \mathbb{Z}^\omega$ by $b_n:=a_{f(n)}$ for all $n\in\omega$. Then it is enough to show that $\mathsf{H}_{\bm{a}}(\I)=\mathsf{H}_{\bm{b}}(\J)$. 

In fact, pick an arbitrary $x\in\mathbb{T}$, and observe that 
\[
\{n\in\omega: \|b_nx\|\geq\varepsilon\}=\{n\in\omega: \|a_{f(n)}x\|\geq\varepsilon\}=f^{-1}[\{n\in\omega: \|a_nx\|\geq\varepsilon\}]
\]
for all $\varepsilon>0$. Therefore $x\in \mathsf{H}_{\bm{a}}(\I)$ if and only if $x\in \mathsf{H}_{\bm{b}}(\J)$.

\medskip

\ref{item:2RK} Pick a finite-to-one map $f:\omega\to\omega$ such that 
$f^{-1}[A]\in\J$ if and only if $A\in\I$ for all $A\subseteq\omega$. Fix a $\bm{a} \in \mathbb{Z}^\omega$. As in the previous point, define $\bm{b} \in \mathbb{Z}^\omega$ by $b_n:=a_{f(n)}$ for all $n\in\omega$; similarly, we claim that $\mathsf{H}^\star_{\bm{a}}(\I)=\mathsf{H}^\star_{\bm{b}}(\J)$. 

First, pick $x \in \mathsf{H}^\star_{\bm{a}}(\I)$. Then there exists
$S \in \mathcal{I}$ such that $\lim_{n \in \omega\setminus S}a_nx=0$. Let
$T:=f^{-1}[S]$. Since $S\in\mathcal I$, we have $T\in\mathcal J$. 
We claim that $\lim_{n\in\omega\setminus T}b_nx=0$. To this aim, fix $\varepsilon>0$. There is
$M\in\omega$ such that $\|a_mx\|<\varepsilon$ for all $m\in\omega\setminus S$ with
$m\geq M$. Since $f$ is finite-to-one, the set $f^{-1}[\{0,\ldots,M-1\}]$ is finite.
Thus there is $N\in\omega$ such that $f(n)\geq M$ for all $n\geq N$. Now, if
$n\in\omega\setminus T$ and $n\geq N$, then $f(n)\notin S$ and $f(n)\geq M$, hence 
$\|b_nx\|=\|a_{f(n)}x\|<\varepsilon$.  
This proves that $x\in \mathsf{H}^\star_{\bm{b}}(\J)$. Hence $\mathsf{H}^\star_{\bm{a}}(\I)\subseteq \mathsf{H}^\star_{\bm{b}}(\J)$. 

Conversely, pick $x\in \mathsf{H}^\star_{\bm{b}}(\J)$. Then there exists
$T\in\mathcal J$ such that $\lim_{n\in\omega\setminus T}b_nx=0$. Define 
$$
S:=\{m\in\omega:f^{-1}[\{m\}]\subseteq T\}.
$$ 
Observe that $f^{-1}[S]\subseteq T \in \mathcal{J}$, hence $S\in\mathcal I$ by the choice of $f$. We claim that $\lim_{m\in\omega\setminus S}a_mx=0$. Fix $\varepsilon>0$. Since
$\lim_{n\in\omega\setminus T}b_nx=0$, there is $N\in\omega$ such that
$\|b_nx\|<\varepsilon$ for all $n\in\omega\setminus T$ with $n\geq N$. Let
$E:=f[\{0,\ldots,N-1\}]$, and choose $M\in\omega$ such that $E\subseteq\{0,\ldots,M-1\}$.
Now, fix $m\in\omega\setminus S$ with $m\geq M$. Since $m\notin S$, there is
$n\in f^{-1}[\{m\}]\setminus T$. Since $m\notin E$, necessarily $n\geq N$. Hence 
$
\|a_mx\|=\|a_{f(n)}x\|=\|b_nx\|<\varepsilon.
$ 
This proves that $\lim_{m\in\omega\setminus S}a_mx=0$. Since $S\in\mathcal I$, we
obtain $x\in \mathsf{H}^\star_{\bm{a}}(\I)$. 
Therefore $\mathsf{H}^\star_{\bm{a}}(\I)=\mathsf{H}^\star_{\bm{b}}(\J)$, which completes the proof. 
\end{proof}

\medskip

\begin{proof}
[Proof of Corollary \ref{RB:corollary}] 
    The first inclusion $\mathsf{Sub}_{\le \omega}\subseteq \mathscr{H}(\mathrm{Fin})$ is well known \cite{MR1877772}.
    For the second one, it follows from Talagrand's characterization of meager ideals that $\mathrm{Fin} \le_{\mathrm{RB}} \mathcal{I}$, see \cite[Theorem 2.1]{MR579439},  
    hence also $\mathrm{Fin} \le_{\mathrm{RK}} \mathcal{I}$. Taking into account that $\mathrm{Fin}$ is a $P$-ideal, we conclude by Theorems \ref{thm:Pproperty} and \ref{thm:RKinclusion} that $\mathscr{H}^\star(\mathrm{Fin})=\mathscr{H}(\mathrm{Fin})\subseteq \mathscr{H}(\mathcal{I})\cap  \mathscr{H}^\star(\mathcal{I})$.
\end{proof}

\medskip

\begin{proof}
[Proof of Lemma \ref{lem:basicIHa}] 
    \ref{item:1basiclemma}. It follows from the fact that $N_{\bm{a}}(\alpha,\varepsilon)=N_{\bm{b}}(\alpha,\varepsilon)\cap \mathrm{ran}(\bm{a})$ for all $\alpha \in \mathbb{T}$ and $\varepsilon>0$. 

    \ref{item:2basiclemma}. It follows from Definition \ref{def:IHa}. 

    \ref{item:3basiclemma}. It follows from item \ref{item:2basiclemma}.
\end{proof}

\medskip

\begin{proof}
[Proof of Theorem \ref{thm:beigblock_representation}] 
Set $Z:=\mathrm{ran}(\bm{a})$ and notice that a point $x \in \mathbb{T}$ belongs to $\mathsf{H}_{Z}(\mathcal{I}_{H,\bm{a}})$ if and only if $N_{\bm{a}}(x,\varepsilon)\in \mathcal{I}_{H,\bm{a}}$ for all $\varepsilon>0$. 

Observe that 
\begin{equation}\label{eq:equalityrepresentations}
\mathsf{H}_{\bm{a}}(\mathcal{J}_{H,\bm{a}})=\mathsf{H}_{Z}(\mathcal{I}_{H,\bm{a}}).
\end{equation}
Indeed, for every \(x\in\mathbb T\) and every \(\varepsilon>0\), we have $
\{n\in\omega:\|a_nx\|\geq\varepsilon\}\in\mathcal J_{H,\bm a}$ 
if and only if $\{a_n: \|a_nx\|\geq\varepsilon\}\in\mathcal I_{H,\bm a}$, and the latter set is exactly $N_{\bm a}(x,\varepsilon)=\{z\in Z:\|zx\|\geq\varepsilon\}$. 

Hence it will be enough to show that $H=\mathsf{H}_{Z}(\mathcal{I}_{H,\bm{a}})$.

First, pick $x\in H$. 
Then, for every $\varepsilon>0$, the set $N_{\bm{a}}(x,\varepsilon)$ belongs to $\mathcal{I}_{H,\bm{a}}$ by definition. 
Hence $x\in \mathsf{H}_{Z}(\mathcal{I}_{H,\bm{a}})$, and therefore $H\subseteq \mathsf{H}_{Z}(\mathcal{I}_{H,\bm{a}})$.

The converse inclusion is obvious if $H=\mathbb{T}$. Otherwise, fix $x\in\mathbb T\setminus H$. 
To conclude the proof (and, hence, to obtain the desired representation), it will be enough to show that 
$
N_{\bm{a}}(x,\nicefrac14)\notin \mathcal{I}_{H,\bm{a}}.
$ 
In fact, suppose for the sake of contradiction that $N_{\bm{a}}(x,\nicefrac14)\in \mathcal{I}_{H,\bm{a}}$, that is, 
\begin{equation}\label{eq:contralkjlsjg}
\exists_{k,m,N \in \omega}
\exists_{\alpha_0, \ldots,\alpha_{k-1} \in H} 
\forall_{n\in N_{\bm{a}}(x,\nicefrac14)\setminus[-N,N]} 
\exists_{i\in k}\quad  
\|n\alpha_i\|\geq 2^{-m}.
\end{equation}
Observe that $x\notin H$, hence $x\notin \langle \alpha_0,\ldots,\alpha_{k-1}\rangle$. In addition, $Z$ is thick, hence $Z\setminus [-N,N]$ is thick as well. It follows from Theorem \ref{thm:hardapproximation} that there exists $n\in Z\setminus[-N,N]$
such that $\|n\alpha_i\|<2^{-m}$ for every $i\in k$ and $\|nx\|>\nicefrac14$. This provides the desired contradiction with \eqref{eq:contralkjlsjg}. Therefore $H=\mathsf{H}_{Z}(\mathcal{I}_{H,\bm{a}})$, which completes the proof. 
\end{proof}

\medskip

\begin{proof}
[Proof of Lemma \ref{lem:basicIH}] 
    \ref{item:4basiclemma}. By Definition \ref{def:IHa}, $\mathcal{I}_H$ is generated by sets $N_{\bm{u}}(\alpha,\varepsilon)=\{n \in \omega: \|n\alpha\|\ge \varepsilon\}$ with $\alpha \in H$ and $\varepsilon>0$. If $\alpha=0$ then $N_{\bm{u}}(\alpha,\varepsilon)$ is empty and if $\alpha\in H\setminus\{0\}$ then $N_{\bm{u}}(\alpha,\varepsilon)$ admits positive asymptotic density, cf. Example \ref{example:singletonzero}. 

    \ref{item:5basiclemma}. This is clear by Definition \ref{def:IHa}. 

    \ref{item:6basiclemma}. If $x=\nicefrac{p}{q}\in\mathbb{Q}\setminus \mathbb{Z}$ with $\mathrm{gcd}(p,q)=1$, then $\I_{\langle x\rangle}\restriction q\omega$ is isomorphic to $\Fin$, while $\omega\setminus q\omega=\{k\in\omega:\|k x\|\geq\nicefrac{1}{2q}\} \in\I_{\langle x\rangle}$.

    \ref{item:7basiclemma}. If $x\notin\mathbb{Q}$ then $\I_{\langle x\rangle}$ is generated by sets $A_n:=N_{\bm{u}}(x,2^{-n})=\{k\in\omega:\|k x\|\geq 2^{-n}\}$ and each set $A_{n+1}\setminus A_n$ is infinite. Hence $A\in\I_{\langle x\rangle}$ if and only if $A$ intersects finitely many infinite sets $A_{n+1}\setminus A_n$.
\end{proof}

\medskip

\begin{proof}
[Proof of Proposition \ref{prop:HFsigmaIHFsigma}] 
Set $Z:=\mathrm{ran}(\bm{a})$. For each $k,m,N \in \omega$ and 
$K\subseteq \mathbb{T}$, define 
$$
C_{k,m,N}(K):=
\left\{
S\subseteq Z :
\exists_{\alpha_0,\ldots,\alpha_{k-1}\in K}\,
\forall_{n\in S\setminus[-N,N]}\,
\exists_{i\in k}, \,\,
\|n\alpha_i\|\ge 2^{-m}
\right\},
$$
and 
$$
D_{k,m,N}(K):=
\left\{
(S,\alpha_0,\ldots,\alpha_{k-1})\in 2^{Z}\times K^k: 
\forall_{n\in S\setminus[-N,N]}\,
\exists_{i\in k},\,\,
\|n\alpha_i\|\ge 2^{-m}
\right\},
$$
with the convention that $C_{0,m,N}(K):=\emptyset$ and $D_{0,m,N}(K):=\emptyset$. 

Observe that if $K$ is compact, then $D_{k,m,N}(K)$ is compact for each $k,m,N \in \omega$. In fact, suppose $k\neq 0$. Then its complement consists of those $(S,\alpha_0,\ldots,\alpha_{k-1})$ for which there is some $z\in Z$ such that
$z\in S\setminus[-N,N]$ 
and $\|z\alpha_i\|<2^{-m}$ for all $i \in k$. 
For each fixed $z$, the condition $z\in S$ is clopen in $2^{Z}$, and the condition
$$
\|z\alpha_i\|<2^{-m}
$$
is open in $K^k$ for all $i \in k$. Therefore the complement of $D_{k,m,N}(K)$ is open, which proves that $D_{k,m,N}(K)$ is closed.  Since $2^{Z}\times K^k$ is compact, $D_{k,m,N}(K)$ is compact as well. 

\medskip

\ref{item:1complexIH}. Let $(K_p: p \in \omega)$ be an increasing sequence of compact sets in $\mathbb{T}$ for which $H=\bigcup_{p\in \omega} K_p$. At this point, observe that 
$$
\mathcal{I}_{H,\bm{a}}
=\bigcup_{k,m,N,p \in\omega}C_{k,m,N}(K_p).
$$
Next, we claim that each $C_{k,m,N}(K_p)$ is closed. To this aim, fix $k,m,N,p \in \omega$ with $k\neq 0$. Since 
$D_{k,m,N}(K_p) \subseteq 2^{Z}\times K_p^k$ is compact by the above observation, it is enough to notice that each $C_{k,m,N}(K_p)$ is the projection on the first coordinate of the corresponding compact set $D_{k,m,N}(K_p)$. Therefore $C_{k,m,N}(K_p)$ is compact as well. 

This proves that $\I_{H, \bm{a}}$ is an $F_\sigma$ ideal on $Z$. Hence there are hereditary compact sets $(Q_m: m \in \omega)$ in $2^Z$ such that $\mathcal I_{H,\bm a}=\bigcup_{m}Q_m$. Let $\pi$ be as in Lemma \ref{lem:piBorel}. Taking into account that 
$$
\mathcal J_{H,\bm a}
=
\pi^{-1}[\mathcal I_{H,\bm a}]
=
\bigcup_{m\in\omega}\pi^{-1}[Q_m],
$$
we conclude by Lemma \ref{lem:piBorel} that $\mathcal J_{H,\bm a}$ is $F_\sigma$. 
Hence, it follows from \cite[Lemma 1.2]{MR748847} that $\mathcal J_{H,\bm a}$ is a $P^+$-ideal, cf. also \cite[Corollary 2.7]{MR5030363}.

\medskip

\ref{item:2complexIH}. 
Observe that 
$$
\mathcal{I}_{H,\bm{a}}
=\bigcup_{k,m,N \in\omega}C_{k,m,N}(H).
$$
Next, we claim that each $C_{k,m,N}(H)$ is analytic. To this aim, fix $k,m,N \in \omega$ with $k\neq 0$. By the above observation $D_{k,m,N}(\mathbb{T})$ is compact in $2^{Z}\times \mathbb{T}^k$. 
It follows that $C_{k,m,N}(H)$ is the projection on the first coordinate of the corresponding analytic set $D_{k,m,N}(\mathbb{T}) \cap (2^{Z}\times H^k)$. Therefore $C_{k,m,N}(H)$ is analytic. This proves that $\I_{H, \bm{a}}$ is an analytic ideal on $Z$. Taking into account that 
$
\mathcal J_{H,\bm a}
=\pi^{-1}[\mathcal I_{H,\bm a}]$ and that $\pi$ is Borel by Lemma \ref{lem:piBorel}, we conclude that $\mathcal J_{H,\bm a}$ is analytic. 



\medskip

\ref{item:4complexIH}. 
Set $Z:=\mathrm{ran}(\bm{a})$. 
First, suppose that $H$ is finite, with $g:=|H|\ge 1$. Set $Z_g:=Z\cap g\mathbb Z$. Since $Z$ is thick, $Z_g$ is infinite. We claim that
\[
\mathcal I_{H,\bm a}
=
\{S\subseteq Z:S\cap Z_g\in\mathrm{Fin}\}.
\]
Indeed, if \(S\in\mathcal I_{H,\bm a}\), then \(S\) is covered, modulo a
finite set, by finitely many sets \(N_{\bm a}(\alpha_i,\varepsilon)\) with
\(\alpha_i\in H\). Since \(n\alpha_i=0\) for every \(n\in Z_g\), it follows
that \(S\cap Z_g\) is finite. 
Conversely, if \(S\cap Z_g\) is finite, then S is contained,
modulo a finite set, in \(Z\setminus Z_g\). If \(g=1\), then \(Z_g=Z\), so
\(S\) is finite. If \(g\geq 2\), then $Z\setminus Z_g
\subseteq
N_{\bm a}\left(\nicefrac{1}{g},\nicefrac{1}{g}\right)
\in\mathcal I_{H,\bm a}$. 
Thus \(S\in\mathcal I_{H,\bm a}\), and the claim follows. 
Hence \(\mathcal I_{H,\bm a}\) is isomorphic to \(\mathrm{Fin}\) if \(g=1\),
and to \(\mathrm{Fin}\oplus\mathcal P(\omega)\) if \(g\geq2\). In particular,
\(\mathcal I_{H,\bm a}\) is a \(P\)-ideal.  

\smallskip

Thus, let us suppose hereafter that $H$ is infinite. We split the proof into two cases. 
Assume first that $H$ is not finitely generated. 
Then there is a sequence $(x_j:j\in\omega)$ in $H$ which is not contained in any finitely generated subgroup of $H$.
Define 
$$
A_j:=N_{\bm{a}}(x_j,\nicefrac{1}{4})
\qquad \text{ for all }j \in \omega. 
$$
Clearly $A_j\in \I_{H,\bm{a}}$. Suppose for the sake
of contradiction that $\I_{H,\bm{a}}$ is a $P$-ideal. Then there is $A\in \I_{H,\bm{a}}$ such that
$A_j\setminus A$ is finite for every $j\in\omega$. 
By Definition \ref{def:IHa}, there are $k,m,N\in\omega$ and $\alpha_0,\ldots,\alpha_{k-1}\in H$ such that, for every $n\in A\setminus[-N,N]$, there is $i\in k$ with $\|n\alpha_i\|\geq 2^{-m}$. 
Pick $j\in\omega$ such that $x_j\notin\langle\alpha_0,\ldots,\alpha_{k-1}\rangle$. Since
$Z\setminus([-N,N]\cup(A_j\setminus A))$ is thick, by Theorem \ref{thm:hardapproximation} there exists $n\in Z\setminus([-N,N]\cup(A_j\setminus A))$ such that 
$$
\|n\alpha_i\|<2^{-m}
\quad\text{for every }i=1,\ldots,k
\quad \text{ and }\quad 
\|nx_j\|>\nicefrac{1}{4}.
$$
Hence $n\in A_j$, and since
$n\notin A_j\setminus A$, we have $n\in A$. This contradicts the choice of
$\alpha_0,\ldots,\alpha_{k-1}$, $m$, and $N$.

It remains to consider the case where $H$ is an infinite subgroup which is finitely generated. 
Pick distinct $g_0,\ldots,g_{d-1} \in H$ such that $H=\langle g_0,\ldots,g_{d-1}\rangle$. Define 
$$
B_m:=\{n\in Z:\exists_{i\in d}\ \|ng_i\|\geq 2^{-m}\}
\qquad \text{ for all }m \in \omega.
$$
Then $B_m\in \I_{H,\bm{a}}$ for every $m\in\omega$. Suppose, similarly to before, that $\I_{H,\bm{a}}$ is a $P$-ideal, and choose $B\in \I_{H,\bm{a}}$ such that $B_m\setminus B$ is finite for every $m\in\omega$. 
There are
$r,M,N\in\omega$ and $\beta_0,\ldots,\beta_{r-1}\in H$ such that, for every
$n\in B\setminus[-N,N]$, there is $\ell\in r$ with $\|n\beta_\ell\|\geq2^{-M}$.
For each $\ell\in r$, write $\beta_\ell=\sum_{i\in d}c_{\ell,i}g_i$, with
$c_{\ell,i}\in\mathbb Z$. Choose $\delta>0$ such that, whenever
$\max_{i\in d}\|ng_i\|<\delta$, we have $\|n\beta_\ell\|<2^{-M}$ for every
$\ell\in r$.

Let $y:=(g_0,\ldots,g_{d-1})\in\mathbb T^d$ and set 
$C:=\overline{\{ny:n\in\mathbb Z\}}$. Since $H$ is infinite, the compact subgroup $C$
is infinite and hence $0$ is not isolated in $C$. Therefore, there is $m\in\omega$
such that the relatively open set
$$
U:=\left\{z\in C:2^{-m}<\max_{i\in d}\|z_i\|<\delta\right\}
$$
is nonempty. 
We claim that $R:=\{n\in\mathbb Z:ny\in U\}$ is syndetic. We proceed as in the proof of Theorem \ref{thm:hardapproximation}: since $U$ is a nonempty relatively open subset of $C$, choose $p\in\mathbb Z$ and an open neighbourhood $W$ of $0$ in $C$ such that $py+W\subseteq U$. 
Since $\{ny:n\in\mathbb Z\}$ is dense in $C$, the family $\{qy+W:q\in\mathbb Z\}$ covers $C$. 
By compactness, there are $q_0,\ldots,q_{\ell-1}\in\mathbb Z$ such that $C\subseteq \bigcup_{j\in \ell}(q_jy+W)$. 
Fix $q\in\mathbb Z$. Then $qy\in q_jy+W$ for some $j\in \ell$, and so $(q-q_j)y\in W$. Hence $(p+q-q_j)y\in py+W\subseteq U$, that is, $p+q-q_j\in R$. Therefore $q\in R+\{q_j-p:j<\ell\}$.  
Thus, $\mathbb Z=R+\{q_j-p:j<\ell\}$ and $R$ is syndetic.

Since $Z$ is thick, it follows that $Z\cap R\neq\emptyset$. 
Moreover, $Z\cap R$ is infinite, because deleting finitely many points from $Z$ leaves a thick set. 
%
Now, for every $n\in Z\cap R$, we have $n\in B_m$ and
$\max_{i\in d}\|ng_i\|<\delta$, thus $\|n\beta_\ell\|<2^{-M}$ for every
$\ell\in r$. Hence, for all but finitely many such $n$ we have
$n\notin B$. Thus $B_m\setminus B$ is infinite, contradicting the choice of $B$.

Therefore $\I_{H,\bm{a}}$ is not a $P$-ideal whenever $H$ is infinite.

\medskip

\ref{item:5complexIH}. This follows from item \ref{item:4complexIH} and Lemma \ref{lem:charactePideal}, by choosing the finite-to-one map $f: \omega\to Z$ given by $f(n):=a_n$ for all $n \in \omega$. 
\end{proof}

\medskip

\begin{proof}
[Proof of Corollary \ref{cor:notPolishable}]
    It follows from \cite[Theorem 2]{MR2388789} that there exists an $F_\sigma$ subgroup $H$ of $\mathbb{T}$ which is not Polishable, cf. also \cite[Theorem 1.20]{MR3461178}. Thanks to Proposition \ref{prop:HFsigmaIHFsigma}, $\mathcal{I}_H$ is an $F_\sigma$ ideal on $\mathrm{ran}(\bm{u})=\omega$. We conclude by Theorem \ref{thm:beigblock_representation} that $H$ is $\mathcal{I}_H$-characterized. 
\end{proof}

\medskip

\begin{proof}
[Proof of Proposition \ref{prop:basiccase}] 
    The \textsc{If} part is obvious. To show the \textsc{Only If} part, pick a sequence $\bm{a}$ with  $\mathrm{supp}(\bm{a})$ infinite and suppose for the sake of contradiction that $\mathsf{H}_{\bm{a}}(\mathrm{Fin})=\mathbb{T}$. It follows from Lemma \ref{lem:almostobvious} that $\bm{a}$ cannot be bounded. Since $\bm{a}$ is unbounded, it is possible to find a suitable subsequence $\bm{b}$ such that $\lim_n b_{n+1}/b_n=\infty$. It follows from Lemma \ref{lem:obviousinclusion} that $\mathsf{H}_{\bm{b}}(\mathrm{Fin})=\mathbb{T}$. Lastly, pick $c \in \mathbb{T}\setminus \{0\}$ and let $(t_n: n\in \omega)\in \mathbb{T}^\omega$ be the constant sequence $(c,c,\ldots)$. It follows from Proposition \ref{prop:interpolation} that there exists $x \in \mathbb{T}$ such that $\lim_nb_nx=c$. However, this contradicts that $\mathsf{H}_{\bm{b}}(\mathrm{Fin})=\mathbb{T}$. 
\end{proof}

\medskip

\begin{proof}
[Proof of Proposition \ref{prop:necessaryfullT}] 
Fix a subset $S\subseteq \mathrm{supp}(\bm{a})$ with $S\in\mathcal I^+$, and let $(q_n: n \in \omega)$ be the increasing enumeration of $(|a_n|: n \in S)$. 
Observe that, for every positive integer $q$, the set
$A_q:=\{n\in\omega: |a_n|=q\}$ belongs to $\mathcal I$. Indeed, if $A_q\in\mathcal I^+$,
then, letting $x:=1/(2q)$, we would have $\|a_nx\|=\nicefrac{1}{2}$ for every $n\in A_q$. Hence
$A_q\subseteq\{n\in\omega:\|a_nx\|\geq \nicefrac{1}{4}\}\notin\mathcal I$, contradicting
$x\in\mathsf H_{\bm a}(\mathcal I)=\mathbb T$. 

In particular, $\{|a_n|:n\in S\}$ is infinite. 
Now, let us suppose for the sake of contradiction that $\liminf_j q_{j+1}/q_j>1$. Then there are
$\lambda>1$ and $j_0\in\omega$ such that $q_{j+1}/q_j\geq\lambda$ for all $j\geq j_0$.
Observe that 
$$
S_1:=\{n\in S: |a_n|=q_j\text{ for some }j\geq j_0\}\in \mathcal{I}^+,
$$
and pick $r\geq1$ such that $\lambda^r>5$. 
Define 
$$
R_s:=\{n\in S_1: |a_n|=q_{j_0+s+tr}\text{ for some }t\in\omega\}
\qquad \text{ for all }s \in r.
$$
Since
$S_1=\bigcup_{s \in r}R_s$, there is $s \in r$ such that $R_s\in\mathcal I^+$.

At this point, define $b_t:=q_{j_0+s+tr}$ for all $t\in\omega$, so that $b_{t+1}/b_t\geq\lambda^r>5$ for
all $t\in\omega$. 
To complete the proof, it will be enough to show that that there exists $x\in\mathbb T$ such that
$\|b_tx\|\geq \nicefrac{1}{4}$ for every $t\in\omega$ (which would contradict $x\in\mathsf H_{\bm a}(\mathcal I)=\mathbb T$ as before). 

To this end, define $I_0:=[1/(4b_0),1/(2b_0)]$. Suppose that $I_t$ has been defined for some $t \in \omega$ so that $|I_t|=1/(4b_t)$ and $\|b_ty\|\geq \nicefrac{1}{4}$ for all $y\in I_t$. 
Since
$|I_t|=1/(4b_t)>5/(4b_{t+1})$, the interval $I_t$ contains some interval of the form
$[(\ell+\nicefrac{1}{4})/b_{t+1},(\ell+\nicefrac{1}{2})/b_{t+1}]$, with $\ell\in\mathbb Z$. Let this interval
be $I_{t+1}$. It follows by construction that $(I_k: k \in \omega)$ is a decreasing sequence of compact intervals of $\mathbb{T}$. By compactness it is possible to choose $x \in \bigcap_k I_k$. Therefore $\|b_tx\|\geq \nicefrac{1}{4}$ for every $t\in\omega$. This completes the proof. 
\end{proof}

\medskip

\begin{proof}
[Proof of Theorem \ref{thm:1P^-}]
The \textsc{If} part is obvious. To show the \textsc{Only If} part, suppose that $S:=\mathrm{supp}(\bm a)\in\mathcal I^+$. We shall prove that $\mathsf H_{\bm a}(\mathcal I)\neq\mathbb T$. There are two cases.

\medskip

\textsc{Case 1.} 
Assume that $\{n\in\omega:a_n=k\}\in\mathcal I^+$ for some $k\in\mathbb Z\setminus\{0\}$. 
Define $x:=\nicefrac{1}{2k}\in\mathbb T$. 
Then, for every $n\in\omega$ with $a_n=k$, we have $\|a_nx\|=\nicefrac{1}{2}$. 
Hence $\{n\in\omega:\|a_nx\|\geq \nicefrac{1}{4}\}$ contains $\{n\in\omega:a_n=k\}\in\mathcal I^+$. 
Therefore $x\notin\mathsf H_{\bm a}(\mathcal I)$, so  $\mathsf H_{\bm a}(\mathcal I)\neq\mathbb T$. 

\medskip

\textsc{Case 2.} 
Assume that $\{n\in\omega:a_n=k\}\in\mathcal I$ for every
$k\in\mathbb Z\setminus\{0\}$. 
Observe that we have $\{n\in S:|a_n|<m\} \in \I$ for every $m\in\omega$. 
Thus, with respect to the restricted ideal $\mathcal{J}:=\mathcal I\restriction S$, we have $\mathcal{J}\text{-}\lim_{n\in S}|a_n|=\infty$. 

Notice that $\mathcal{ED}\not\leq_{\mathrm{K}}\mathcal{J}\restriction A$ for every $A\in\mathcal{J}^+$.  
Applying Lemma \ref{lem:P^-} to $\mathcal J=\mathcal{I}\upharpoonright S$ and to the sequence $(a_n:n\in S)$, we obtain a sequence $(k(n):n\in\omega)$ of elements of $S$ such that $\{k(n):n\in\omega\}\in\mathcal I^+$ and, after possibly replacing $\bm a$ by $-\bm a$ on this set, the integers $a_{k(n)}$ are positive and
$$
\lim_{n\to \infty}\frac{a_{k(n+1)}}{a_{k(n)}}=\infty.
$$
Equivalently, the positive sequence $b_n:=|a_{k(n)}|$ satisfies $\lim_n b_{n+1}/b_n=\infty$. 

By Proposition \ref{prop:interpolation}, applied to the sequence $(b_n:n\in\omega)$ and to the constant sequence $\nicefrac{1}{2}$, there exists $x\in\mathbb T$ such that $\lim_n \|b_nx-\nicefrac{1}{2}\|=0$. 
Hence $\lim_n \|a_{k(n)}x\|= \nicefrac{1}{2}$. In particular, there exists a finite set $F \in \mathrm{Fin}$ such that $\{n\in\omega:\|a_nx\|\geq \nicefrac{1}{4}\}$ contains $\{k(n):n\in\omega\}\setminus F \in \mathcal{I}^+$. 
Therefore $x\notin\mathsf H_{\bm a}(\mathcal I)$, and again $\mathsf H_{\bm a}(\mathcal I)\neq\mathbb T$.
\end{proof}

\medskip

\begin{proof}
[Proof of Corollary \ref{cor:1P^-A}]
Let $\mathcal I$ be a nowhere tall ideal. 
Let us show that $\mathcal{I}$ satisfies \eqref{eq:technicalP-EDfin}. 
To this aim, fix a set $A_0\in\mathcal I^+$. 
Since $\mathcal I$ is nowhere tall, there is an infinite set $B\subseteq A_0$ such that $B\in\I^+$ and $\mathcal I\restriction B=[B]^{<\omega}$. 
Suppose for the sake of contradiction that $f:A_0\to\omega^2$ witnesses $\mathcal{ED}\leq_{\mathrm{K}}\mathcal I\restriction A_0$.
If $f[B]$ is finite, then $f[B]\in\mathcal{ED}$ and $B\subseteq f^{-1}[f[B]]$, contradicting $f^{-1}[f[B]]\in\mathcal I\restriction A_0$. 
If $f[B]$ is infinite, then, since $\mathcal{ED}$ is tall, there is an infinite set $D\subseteq f[B]$ such that $D\in\mathcal{ED}$. 
Then $f^{-1}[D]\cap B$ is infinite,  so $f^{-1}[D]\notin\mathcal I\restriction A_0$, again a contradiction. 
Hence $\mathcal{ED}\not\leq_{\mathrm{K}} \mathcal I\restriction A_0$.

The conclusion follows from Theorem \ref{thm:1P^-}.
\end{proof}

\medskip

\begin{proof}
[Proof of Corollary \ref{cor:1P^-B}]
 It is known that both $\mathsf{nwd}$ and $\mathsf{null}$ are homogeneous ideals, that is, for every $A\in \mathsf{nwd}^+$, the ideal $\mathsf{nwd}\restriction A$ is isomorphic to $\mathsf{nwd}$, and for
every $A\in \mathsf{null}^+$ the ideal $\mathsf{null}\restriction A$ is isomorphic to $\mathsf{null}$, see \cite{MR1955288}. 
 Moreover, $\mathcal{ED}\not\leq_{\mathrm{K}} \mathsf{nwd}$ and $\mathcal{ED}\not\leq_{\mathrm{K}} \mathsf{null}$, see \cite[Fig. 1]{MR3600759}. Hence both $\mathsf{nwd}$ and $\mathsf{null}$ satisfy condition \eqref{eq:technicalP-EDfin}. The conclusion follows from Theorem \ref{thm:1P^-}. 
\end{proof}

\medskip

\begin{proof}
[Proof of Theorem \ref{thm:2P^-}]
Hereafter, for each $n\in\mathbb Z\setminus\{0\}$ let $\nu_2(n)\in\omega$ be the
$2$-adic valuation of $n$, that is, the largest $k \in \omega$ such that $2^k$ divides
$n$. Now, fix an arbitrary integer sequence $\bm a\in\mathbb Z^\omega$. We need to show that
$$
\mathsf H_{\bm b}(\J)\neq\mathsf H_{\bm a}(\I).
$$

Let us suppose for the sake of contradiction that $\mathsf H_{\bm b}(\J)=\mathsf H_{\bm a}(\I)$. Observe also that $\mathsf H_{\bm b}(\J)\neq\mathbb T$,
since $\nicefrac13\notin\mathsf H_{\bm b}(\J)$. Indeed, $\|\nicefrac{b_n}{3}\|=\nicefrac{1}{3}$ for every
$n\in\omega$. Then also $\mathsf H_{\bm a}(\I)\neq \mathbb{T}$. It follows by the standing hypothesis on $\I$ and Theorem \ref{thm:1P^-} that $\mathrm{supp}(\bm a)\in\I^+$. 

First, suppose that there is
$k\in\omega$ such that
$\{n\in\mathrm{supp}(\bm a):\nu_2(a_n)=k\}\in\I^+$, and define $x:=2^{-k-1}$. Then
$x\notin\mathsf H_{\bm a}(\I)$, because
\[
\{n\in\omega:\|a_nx\|\geq\nicefrac12\}\supseteq
\{n\in\mathrm{supp}(\bm a):\nu_2(a_n)=k\}\in\I^+.
\]
At the same time, $x\in\mathsf H_{\bm b}(\J)$, since
$\{n\in\omega:\|2^nx\|\geq\varepsilon\}$ is finite for every $\varepsilon>0$.
Hence $\mathsf H_{\bm b}(\J)\neq\mathsf H_{\bm a}(\I)$ in this case.

Thus, we can assume that $\mathrm{supp}(\bm a)\in\I^+$ and
$\{n\in\mathrm{supp}(\bm a):\nu_2(a_n)=k\}\in\I$ for every $k\in\omega$. Since
$\mathcal{ED}\not\leq_{\mathrm{K}}\I\restriction \mathrm{supp}(\bm a)$, there is $A\subseteq\mathrm{supp}(\bm a)$ such that $A\in\I^+$ and $|\{n\in A:\nu_2(a_n)=k\}|\leq 1$ for every $k\in\omega$.

Since $\J\not\le_{\mathrm{K}}\I\restriction A$, the map $f:A\to\omega$ given by $f(n):=\nu_2(a_n)$ does not witness $\J\le_{\mathrm{K}}\I\restriction A$. 
Hence there is $B\subseteq A$ such that $B\in\I^+$ and $f[B]\in\J$. 
Let $(k(n):n\in\omega)$ be an enumeration of $B$ such that $f(k(n))\leq f(k(n+1))$ for every $n\in\omega$.

Next, we define an increasing sequence $(j_i:i\in\omega)$ in $\omega$. Let $j_0:=0$.
Suppose that $j_i$ has been defined for some $i \in \omega$. 
Since $\{n\in\omega:f(k(n))\leq j_i\}$ is finite, choose $j_{i+1}>j_i$ such that $j_{i+1}=f(k(m))$ for some $m\in B$ and
$$
|a_{k(n)}|\sum_{\ell>j_{i+1}}\frac{1}{2^{\ell+1}}<\frac14
$$
for every $n\in\omega$ with $f(k(n))\leq j_i$. 
Next, define
$$
F_i:=\{k(n):j_{i-1}<f(k(n))\leq j_i\}
\quad \text{ for every }i\in \omega,
$$
where $j_{-1}:=-1$ by convention. 
Then $(F_i:i\in\omega)$ is a partition of $B$ into finite sets. Since $\mathcal{ED}\not\le_{\mathrm{K}}\I\restriction B$, there is $C\subseteq B$ such that $C\in\I^+$ and $|C\cap F_i|=1$ for every $i\in\omega$. 

Therefore either $C\cap\bigcup_i F_{2i}\in\I^+$ or $C\cap\bigcup_i F_{2i+1}\in\I^+$. 
Without loss of generality, assume that 
$$
D:=C\cap\bigcup_{i\in\omega}F_{2i}\in\I^+.
$$
For every $t\in\omega$, let $m(t)$ be such that
$D\cap F_{2t}=\{k(m(t))\}$.

Define
$$
x:=\sum_{t\in\omega}\frac{1}{2^{f(k(m(t)))+1}}.
$$
First, we will show that $x\notin\mathsf H_{\bm a}(\I)$. Fix $t\in\omega$. Since
$k(m(t))\in F_{2t}$, we have $f(k(m(t)))\leq j_{2t}$. Moreover, if $s>s'>t$, then
$f(k(m(s)))>f(k(m(s')))>j_{2t+1}$. Hence
$$
|a_{k(m(t))}|\sum_{s>t}\frac{1}{2^{f(k(m(s)))+1}}
\leq
|a_{k(m(t))}|\sum_{\ell>j_{2t+1}}\frac{1}{2^{\ell+1}}
<\frac14.
$$
Also, $a_{k(m(t))}\sum_{s<t}2^{-f(k(m(s)))-1}$ is an integer, while
$\|a_{k(m(t))}2^{-f(k(m(t)))-1}\|=\nicefrac12$. Therefore $\|a_{k(m(t))}x\|\geq\nicefrac14$ for every $t\in\omega$. It follows that
$$
\{n\in\omega:\|a_nx\|\geq\nicefrac14\}\supseteq D\in\I^+,
$$
so $x\notin\mathsf H_{\bm a}(\I)$.

It remains to show that $x\in\mathsf H_{\bm b}(\J)$. Note that if $n\in(f(k(m(t-1))),f(k(m(t)))]$ for some $t\in\N$, then 
$$\|2^nx\|=\left\|2^{n-f(k(m(t)))-1} + 2^n\sum_{s>t}\frac{1}{2^{f(k(m(s)))+1}}\right\|$$
and 
$$2^n\sum_{s>t}\frac{1}{2^{f(k(m(s)))+1}}<2^{n-f(k(m(t)))-2}. $$
Therefore, for such $n$ and any $i\in\omega$ we can have $\|2^nx\|\geq 2^{-i}$  only if $f(k(m(t)))-n<i$. 
Since the sequence
$(f(k(m(t))):t\in\omega)$ is strictly increasing and has gaps at least $2$, 
it follows that 
$$
\{n\in\omega:\|2^nx\|\geq2^{-i}\}\subseteq\bigcup_{j \in i}(f[D]-j)
\quad \text{ for all }i\ge 1,
$$
where as usual $f[D]-j:=\{n\in\omega:n+j\in f[D]\}$. Since $f[D]\subseteq f[B]\in\J$ and
$\J$ is translation invariant, the set $\bigcup_{j\in i}(f[D]-j)$ belongs to $\J$ for
every $i\in\omega$. Hence
$\{n\in\omega:\|2^nx\|\geq2^{-i}\}\in\J$ for every $i\in\omega$, which proves that
$x\in\mathsf H_{\bm b}(\J)$.

Thus $\mathsf H_{\bm b}(\J)\neq\mathsf H_{\bm a}(\I)$ for every $\bm a\in\mathbb Z^\omega$. 
Therefore $\mathsf H_{\bm b}(\J)\notin\mathscr H(\I)$.
\end{proof}

\begin{remark}\label{rmk:notinclusion}
    An inspection of the proof of Theorem \ref{thm:2P^-} shows that something stronger holds, namely, $\mathsf{H}_{\bm{b}}(\mathcal{J})$ is not contained in any \emph{proper} $\mathcal{I}$-characterized subgroup of $\mathbb{T}$. 
\end{remark}

\begin{proof}
[Proof of Corollary \ref{cor:Hpowernotfin}] 
By the proof of Corollary \ref{cor:1P^-A}, $\mathcal{I}=\mathrm{Fin}$ is an ideal which satisfies \eqref{eq:technicalP-EDfin}. In addition, if $\mathcal{J}$ is tall then $\mathcal{J}\not\le_{\mathrm{K}} \mathrm{Fin}$. Since $\mathrm{Fin}\restriction A$ is isomorphic to $\mathrm{Fin}$, for every $A\in\mathrm{Fin}^+$, we have $\mathcal{J}\not\le_{\mathrm{K}} \mathrm{Fin}\restriction A$, for every $A\in\mathrm{Fin}^+$. The conclusion follows from Theorem \ref{thm:2P^-}. 
\end{proof}

\medskip

\begin{proof}
[Proof of Corollary \ref{cor:caseZ}] 
As observed in the proof of Corollary \ref{cor:1P^-B} , both $\mathcal{I}=\mathsf{nwd}$ and $\mathcal{I}=\mathsf{null}$ are homogeneous ideals which satisfy \eqref{eq:technicalP-EDfin}. In addition, $\mathcal{J}=\mathcal{Z}$ is a tall translation invariant ideal, and we have $\mathcal{Z}\not\leq_{\mathrm{K}} \mathsf{nwd}$ and $\mathcal{Z}\not\leq_{\mathrm{K}} \mathsf{null}$ by \cite[Fig. 1]{MR3600759}. By homogeneity, condition \eqref{eq:secondtechnicalcondition} also  holds. 
The conclusion follows from Theorem \ref{thm:2P^-}. 
\end{proof}

\medskip

\begin{proof}
    [Proof of Corollary \ref{cor:meager}] 
    The inclusion $\mathscr{H}(\mathrm{Fin})\subseteq \mathscr{H}(\mathcal{J})$ follows from Corollary \ref{RB:corollary}. On the other hand,  $\mathsf{H}_{\bm{b}}(\mathcal{J})\notin \mathscr{H}(\mathrm{Fin})$ by Corollary \ref{cor:Hpowernotfin}. Hence $\mathscr{H}(\mathrm{Fin})\neq \mathscr{H}(\mathcal{J})$. 
\end{proof}



\bibliographystyle{amsplain}

\end{document}